\documentclass[reqno]{amsart}
\usepackage{enumerate}
\usepackage{tabto}
\usepackage[mathscr]{euscript}
\usepackage{xcolor}
\usepackage{layout}
\usepackage{fancyhdr}
\usepackage{array}
\usepackage{amsfonts}
\usepackage{amsmath}
\usepackage{amssymb}
\usepackage{mathtools}
\usepackage{graphicx}
\usepackage{bm}
\usepackage{enumitem}
\usepackage{caption} 
\usepackage{color}
\usepackage{kotex}
\usepackage{csquotes}
\usepackage{bookmark}
\usepackage{float}
\usepackage{multirow}
\usepackage{slashbox}
\usepackage[square,numbers,sort&compress]{natbib}
\usepackage{hyperref}
\hypersetup{colorlinks=true,linkcolor=blue,citecolor=red}
\allowdisplaybreaks

\def\Xint#1{\mathchoice
{\XXint\displaystyle\textstyle{#1}}%
{\XXint\textstyle\scriptstyle{#1}}%
{\XXint\scriptstyle\scriptscriptstyle{#1}}%
{\XXint\scriptscriptstyle\scriptscriptstyle{#1}}%
\!\int}
\def\XXint#1#2#3{{\setbox0=\hbox{$#1{#2#3}{\int}$ }
\vcenter{\hbox{$#2#3$ }}\kern-.6\wd0}}

\def\dashint{\Xint-}

\newtheorem{theorem}{Theorem}[section]
\newtheorem{lemma}[theorem]{Lemma}
\newtheorem{corollary}[theorem]{Corollary}
\newtheorem{proposition}[theorem]{Proposition}
\newtheorem{remark}[theorem]{Remark}
\theoremstyle{definition}
\newtheorem{definition}[theorem]{Definition}

\numberwithin{equation}{section}

\newcommand{ \mr }{ \mathbb{R} }

\newcommand{\iintss}{{\int\hspace{-0.28cm}\int}}
\newcommand{\miints}{{\iintss\hspace{-0.56cm} -\hspace{-0.15cm}-}}
\newcommand{\miint}[1]{{\miints_{\hspace{-0.13cm}#1}}}

\begin{document}
\title[Parabolic Lipschitz truncation for multi-phase problems]{Parabolic Lipschitz truncation for multi-phase problems: the degenerate case}

\author{Bogi Kim}\address{Department of Mathematics, Kyungpook National University, Daegu, 41566, Republic of Korea} \email{rlaqhrl4@knu.ac.kr}
\author{Jehan Oh}\address{Department of Mathematics, Kyungpook National University, Daegu, 41566, Republic of Korea} \email{jehan.oh@knu.ac.kr}
\author{Abhrojyoti Sen}\address{Goethe-Universit\"{a}t Frankfurt, Institut f\"{u}r Mathematik, Robert-Mayer-Str. 10, D-60629 Frankfurt, Germany} \email{sen@math.uni-frankfurt.de}

\subjclass{Primary 35B65 ; Secondary 35K65, 35K55, 35D30, 35A01 }
\date{\today.}
\keywords{Lipschitz truncation, parabolic multi-phase problems, parabolic systems, degenerate case, energy estimates}
\thanks{Bogi Kim is supported by the National Research Foundation of Korea (NRF) grant funded by the Korea government [Grant No. RS-2023-00217116]. Jehan Oh is supported by the National Research Foundation of Korea (NRF) grant funded by the Korea government {[Grant No. RS-2025-00555316]}. Abhrojyoti Sen is supported by research grants from the Alexander von Humboldt Foundation for postdocs.}

\begin{abstract}
This article is devoted to exploring the Lipschitz truncation method for parabolic multi-phase problems. The method is based on Whitney decomposition and covering lemmas with a delicate comparison scheme of appropriate alternatives to distinguish phases, as introduced by the first and second authors in \cite{Kim_Oh_2024}.
\end{abstract}
\maketitle

\section{\bf Introduction}
In this paper, we prove energy estimates for weak solutions to parabolic multi-phase problems of type
\begin{align*}
    &u_t-\operatorname{div}(|\nabla u|^{p-2}\nabla u+a(z)|\nabla u|^{q-2}\nabla u + b(z)|\nabla u|^{s-2}\nabla u)\\
    &\qquad\qquad =-\operatorname{div}(|F|^{p-2}F+a(z)|F|^{q-2}F + b(z)|F|^{s-2}F) \quad \text{in }\Omega_T,
\end{align*}
where $\Omega_T:=\Omega\times (0,T)$ represents a space-time cylinder with a bounded open set $\Omega \subset \mr^n$ for $n\geq 2$, $2\leq p <q <s <\infty$ and the modulating coefficients $a(\cdot)$ and $b(\cdot)$ are nonnegative and H\"{o}lder continuous.

Energy estimates are very important for proving the existence and regularity results for partial differential equations. We prove the energy estimates using the Lipschitz truncation method for parabolic multi-phase problems. Here, the Lipschitz truncation is a method of redefining a given function so that it keeps its values on a specific `good set', while redefining it in `bad sets' using a partition of unity related to a Whitney covering argument. Acerbi-Fusco \cite{Acerbi1984,Acerbi1988} have introduced the Lipschitz truncation for elliptic problems. Furthermore, Kinnunen-Lewis \cite{Kinnunen2002} and Kim-Kinnunen-S\"{a}rki\"{o} \cite{Wontae2023a} have studied related methods for the parabolic $p$-Laplace system and for the parabolic double phase systems, respectively.

In this paper, we deal with a parabolic multi-phase system
\begin{equation}    \label{eq: main equation}
    u_t -\operatorname*{div}\mathcal{A}(z,u,\nabla u)=-\operatorname*{div}\mathcal{B}(z,F) \quad \text{in $\Omega_T$,}
\end{equation}
 where Carath\'{e}odory vector fields $\mathcal{A}:\Omega_T\times \mr^N \times\mr^{Nn}\rightarrow \mr^{Nn}$ and $\mathcal{B}:\Omega_T\times \mr^{Nn}\rightarrow \mr^{Nn}$ satisfy the following growth conditions: for any $z \in \Omega_T, \; v\in \mr^N$ and $\xi\in \mr^{Nn}$, there exist two constants $0<\nu\leq L<\infty$ such that 
\begin{equation}    \label{eq: growth condition of A}
    \nu H(z,|\xi|)\leq \mathcal{A}(z,v,\xi)\cdot \xi, \quad |\mathcal{A}(z,v,\xi)||\xi|\leq LH(z,|\xi|)
\end{equation}
and
\begin{equation}    \label{eq: growth condition of B}
    |\mathcal{B}(z,\xi)||\xi|\leq LH(z,|\xi|),
\end{equation}
where the function $H:\Omega_T\times  \mr^+\rightarrow \mr^+$ is defined by 
$$
H(z,\kappa)=\kappa^p+a(z)\kappa^q+b(z)\kappa^s
$$
for $z\in\Omega_T$ and $\kappa\in\mr^+$. Furthermore, the source term $F:\Omega_T\rightarrow \mr^{Nn}$ satisfies 
\begin{equation}    \label{eq: assumption of F}
    \iint_{\Omega_T} H(z,|F|)\,dz < +\infty.
\end{equation}
We assume that the modulating coefficients $a:\Omega_T\rightarrow\mr^+$ and $b:\Omega_T\rightarrow\mr^+$ satisfy
\begin{equation}\label{eq : condition of n,p,q,alpha}
    q\leq p+\frac{2\alpha}{n+2},\quad 0\leq a\in C^{\alpha ,\frac{\alpha}{2}}(\Omega_T)\quad \text{for some }\alpha\in(0,1],
\end{equation}
and
\begin{equation}\label{eq : condition of n,p,s,beta}
    s\leq p+\frac{2\beta}{n+2},\quad 0\leq b\in C^{\beta,\frac{\beta}{2}}(\Omega_T)\quad \text{for some }\beta\in (0,1].
\end{equation}
Here, $a\in C^{\alpha,\frac{\alpha}{2}}(\Omega_T)$ means that $a\in L^\infty(\Omega_T)$ and there exists a H\"{o}lder constant $[a]_\alpha:=[a]_{\alpha,\frac{\alpha}{2};\Omega_T}>0$ such that 
$$
|a(x_1,t_1)-a(x_2,t_2)|\leq [a]_\alpha \left(|x_1-x_2|+\sqrt{|t_1-t_2|}\right)^\alpha
$$
for all $x_1,x_2\in \Omega$ and $t_1,t_2\in (0,T)$.
\begin{definition}\label{weak solution}
    A function $u : \Omega\times (0, T) \to \mathbb{R}^N$ satisfying 
    \begin{align*}
        u \in C(0,T; L^2(\Omega,\mr^N))\cap L^1(0, T; W^{1, 1}(\Omega,\mr^N))
    \end{align*}
    and
    $$
    \iint_{\Omega_T} [H(z,|u|)+H(z,|\nabla u|)]\, dz <\infty
    $$
    is a weak solution to \eqref{eq: main equation} if
    $$
    \iint_{\Omega_T} [-u\cdot \varphi_t +\mathcal{A}(z,u,\nabla u)\cdot \nabla\varphi]\, dz =\iint_{\Omega_T} [\mathcal{B}(z,F)\cdot \nabla\varphi]\, dz
    $$
    for every $\varphi \in C^\infty_0(\Omega_T,\mr^N)$.
\end{definition}

The parabolic multi-phase problems derive from elliptic double phase problems. The elliptic double phase problems of type
$$
-\operatorname{div}(|\nabla u|^{p-2}\nabla u +a(x)|\nabla u|^{q-2}\nabla u)=-\operatorname{div}(|F|^{p-2}F + a(x)|F|^{q-2}F)
$$
was first introduced in \cite{Zhikov1986,Zhikov1993,Zhikov1995,Zhikov1997}. These problems originate from the Lavrentiev phenomenon and the homogenization of strongly anisotropic materials. According to \cite{Colombo2015a,Fonseca2004}, the conditions 
\begin{equation}    \label{eq : condition 1 of p,q,alpha in elliptic Double phase}
    a(\cdot)\in C^\alpha(\Omega), \quad \alpha\in(0,1]\quad \text{and}\quad \frac{q}{p}\leq 1 +\frac{\alpha}{n}
\end{equation}
and
\begin{equation}    \label{eq : condition 2 of p,q,alpha in elliptic Double phase}
    u\in L^\infty (\Omega),\quad a(\cdot)\in C^\alpha (\Omega),\quad  \alpha\in(0,1] \quad \text{and}\quad q\leq p+\alpha
\end{equation}
are sharp for obtaining regularity results of weak solutions. In fact, when \eqref{eq : condition 1 of p,q,alpha in elliptic Double phase} or \eqref{eq : condition 2 of p,q,alpha in elliptic Double phase} holds, the gradient of a weak solution $u$ is H\"{o}lder continuous, see \cite{Baroni2018,Colombo2015,Esposito2004, Colombo2015a}. Moreover, the Calder\'{o}n-Zygmund estimates have been discussed in \cite{Baasandorj2020,DeFilippis2019,Colombo2016}. Also, Baroni-Colombo-Mingione \cite{Baroni2015} have investigated the Harnack's inequality. In addition, other regularity results for elliptic double phase problems have been discussed in \cite{Ok2017,Ok2020,Byun2017,Byun2021,Byun2021a,Byun2020,Haestoe2022,Haestoe2022a,Oh2024a}. The regularity results for elliptic multi-phase problems given by 
$$
\begin{aligned}
    &-\operatorname{div}\left(|\nabla u|^{p-2}\nabla u +\sum_{i=1}^m a_i(x)|\nabla u|^{p_i-2}\nabla u\right)\\
    &\qquad\qquad =-\operatorname{div}\left(|F|^{p-2}F+\sum_{i=1}^m a_i(x)|F|^{p_i-2}F\right)\,\,\, \ \text{in }\, \Omega,
\end{aligned}
\qquad 
$$
with $1<p<p_1\leq \cdots\leq p_m$ and $0\leq a_i(\cdot)\in C^{0,\alpha_i}(\bar{\Omega})$, $\alpha_i\in (0,1]$, have also been discussed in \cite{Baasandorj2021,DeFilippis2022,DeFilippis2019a,Fang2022}.

On the other hand, the regularity for parabolic double phase problems
$$
u_t-\operatorname{div}(|\nabla u|^{p-2} \nabla u + a(z)|\nabla u|^{q-2}\nabla u)=-\operatorname{div}(|F|^{p-2}F +a(z)|F|^{q-2}F)\quad \text{in }\Omega_T
$$
with {$2\leq p<q$} has been studied recently. If \eqref{eq : condition of n,p,q,alpha} holds, the existence of weak solutions has been discussed in \cite{Chlebicks2019}, see also \cite{Wontae2023a,Singer2016}. The gradient higher integrability and the Calder\'{o}n-Zygmund type estimate have been studied in \cite{Wontae2023b,2023_Gradient_Higher_Integrability_for_Degenerate_Parabolic_Double-Phase_Systems}. Moreover, the gradient higher integrability results for degenerate and singular parabolic multi-phase problems have also been studied in \cite{Kim_Oh_2024} and \cite{sen2024}, respectively.

In this paper, our goal is to prove the energy estimates of the weak solution to \eqref{eq: main equation}. For this, we denote parabolic cylinders $U_{r,\tau}(z_0)$ and $Q_r(z_0)$ by
\begin{equation}    \label{eq : definition of U and Q}
U_{r,\tau}(z_0):= B_r(x_0)\times \ell_\tau (t_0) \quad \text{and}\quad Q_r(z_0):=B_r(x_0)\times I_r(t_0)
\end{equation}
with 
\begin{equation}\label{eq : definition of ell and I_r}
    \ell_\tau(t_0):=(t_0-\tau,t_0+\tau)\quad \text{and}\quad I_r(t_0):=(t_0-r^2,t_0+r^2).
\end{equation}
Our main theorem of this paper is as follows.
\begin{theorem} \label{thm : the Caccioppoli inequality}
    Let $u$ be a weak solution to \eqref{eq: main equation}. Then, for $U_{R_2,S_2}(z_0)\subset \Omega_T$, $R_1\in [R_2/2,R_2)$ and $S_1\in[S_2/2^2,S_2)$, there exists a constant $c$ depending on $n,p,q,s,\nu$ and $L$ such that the following inequality holds: 
    \begin{align*}
        &\sup_{t\in(t_0-S_1,t_0+S_1)}\dashint_{B_{R_1}(x_0)} \frac{|u-(u)_{U_{R_1,S_1}(z_0)}|^2}{S_1}\; dx +\miint{U_{R_1,S_1}(z_0)}H(z,|\nabla u|)\;dz\\
        &\qquad \leq c\miint{U_{R_2,S_2}(z_0)}   H\left(z,\frac{|u-(u)_{U_{R_2,S_2}(z_0)}|}{R_2-R_1}\right)\; dz\\
        &\quad\qquad + c\miint{U_{R_2,S_2}(z_0)}\frac{|u-(u)_{U_{R_2,S_2}(z_0)}|^2}{S_2-S_1}\;dz +c\miint{U_{R_2,S_2}(z_0)} H(z,|F|) \;dz.
    \end{align*}
\end{theorem}  
\begin{remark}
In \cite{Kim_Oh_2024}, the energy estimate presented in Lemma 3.1 is derived under the assumption that $|\nabla u| \in L^s(\Omega_T).$ Consequently, the gradient higher integrability result \cite[Theorem 1.2]{Kim_Oh_2024} also relies on this assumption. However, in light of the energy estimate above, Theorem 1.2 in \cite{Kim_Oh_2024} can be established under the weaker assumption specified in Definition \ref{weak solution}.
\end{remark}
\begin{remark}
We would like to mention that the existence and uniqueness results for \eqref{eq: main equation} with Dirichlet boundary condition can be obtained by closely following the proofs of Theorem 2.6 and Theorem 2.7 in \cite{Wontae2023a}. Therefore, we choose not to pursue such aspects in this paper.
\end{remark}
To prove Theorem \ref{thm : the Caccioppoli inequality}, we construct the Whitney decomposition by dividing into $p$-, $(p,q)$-, $(p,s)$- and $(p,q,s)$-phases in Section \ref{sec : Whitney decomposition and covering lemmas}. Furthermore, in this section, we prove the Vitali covering argument by dividing it into a total of $16$ cases as in \cite{Kim_Oh_2024} and establish the related properties, which we summarize in Lemma \ref{LEM2.9}. In Section \ref{sec : Construction of test function via Lipschitz truncation}, we define the Lipschitz truncation and establish the related properties. We then prove the energy estimates in Section \ref{sec : proof of main theorem}.

\section{\bf Whitney decomposition and covering lemmas} \label{sec : Whitney decomposition and covering lemmas}
In this section, we construct a family of Whitney decomposition and show that the family covers the bad set $E(\Lambda)^c$ via a Vitali covering argument. The construction of such a decomposition is heavily used in the subsequent sections.
\subsection{\bf Auxiliary definitions}
Let $z_0=(x_0,t_0)\in \mr^{n+1}$ and $\varrho>0$. Parabolic cylinders with quadratic scaling in time are denoted as
$$
Q_\varrho (z_0)=B_\varrho \times I_\varrho(t_0),
$$
where $B_\varrho=B_\varrho(x_0)=\{y\in\mr^n : |x_0-y|<\varrho\}$ and $I_\varrho(t_0)=(t_0-\varrho^2,t_0+\varrho^2)$.

Let us define the strong maximal function  for $f\in L^1_{loc}(\mathbb{R}^{n+1})$ as 
\begin{align}\label{max fn}
    M f(z)=\sup_{z\in Q}\miint{Q}|f|\, dw,
\end{align}
where the supremum is taken over cubes $Q \subset \mathbb{R}^{n+1}$. { As in \cite{Wontae2023a}, using the Hardy-Littlewood-Wiener maximal function theorem with respect to space and time, we obtain the following lemma.}
\begin{lemma}\label{lem : property of maximal function}
    For $1<\sigma<\infty$ and $f\in L^\sigma (\mr^{n+1})$, there exists a constant $c=c(n,\sigma)$ such that
    $$
    \iint_{\mr^{n+1}} |Mf|^\sigma \, dz\leq c \iint_{\mr^{n+1}}|f|^\sigma\, dz.
    $$
\end{lemma}

{ Without loss of generality we can assume that the modulating coefficient functions $a(\cdot)$ and $b(\cdot)$ are defined in $\mr^{n+1}$ satisfying $a(\cdot)\in C^{\alpha,\frac{\alpha}{2}}(\mr^{n+1})$ and $b(\cdot)\in C^{\beta,\frac{\beta}{2}}(\mr^{n+1})$, see for instance \cite{Wontae2023a} and \cite[Theorem 2.7]{Kinnunen2021}.} Let $f\in L^p(\mr^{n+1})$, $f\geq 0$, be such that
$$
\iint_{\mr^{n+1}} (f^p+a f^q+b f^s)\, dz <\infty,
$$
and let $d=\frac{s-1+p}{2}$. By Lemma \ref{lem : property of maximal function}, there exists $\Lambda_0>1+\|a\|_{L^\infty(\mr^{n+1})}+\|b\|_{L^\infty (\mr^{n+1})}$ such that
$$
\begin{aligned}
\iint_{\mr^{n+1}}\left(M(f^d+(a f^q)^{\frac{d}{p}}+(b f^s)^{\frac{d}{p}})(z)\right)^\frac{p}{d}\, dz &\leq c(n,p,s) \iint_{\mr^{n+1}} (f^p+a f^q+ b f^s)\, dz\\ &\leq \Lambda_0.
\end{aligned}
$$

Let $\Lambda > \Lambda_0$ and we define the set
\begin{align}\label{defn of E}
    E(\Lambda)=\left\{z\in \mathbb{R}^{n+1}: \,\, M\left(f^d+(af^q)^{\frac{d}{p}}+(bf^s)^{\frac{d}{p}}\right)(z)\leq \Lambda^{\frac{d}{p}}\right\}.
\end{align}
Chebyshev's inequality implies that
\begin{equation}\label{limit of Λ|E(Λ)^c|}
    \lim_{\Lambda\rightarrow \infty} \Lambda|E(\Lambda)^c|\leq \lim_{\Lambda \rightarrow \infty}\iint_{E(\Lambda)^c}\left(M(f^d+(af^q)^\frac{d}{p}+(bf^s)^\frac{d}{p})(z)\right)^\frac{p}{d}\, dz=0.
\end{equation}
Now, we write $K\geq 2$ as
\begin{align}\label{defn of K}
    K:=2&+800[a]_{\alpha}\left(\frac{1}{|B_1|}\iint_{\mathbb{R}^{n+1}}\left(M(f^d+(a f^q)^{\frac{d}{p}}+(bf^s)^{\frac{d}{p}})\right)^{\frac{p}{d}}(z) \,dz\right)^{\frac{\alpha}{n+2}}\nonumber\\
    &+800[b]_{\beta}\left(\frac{1}{|B_1|}\iint_{\mathbb{R}^{n+1}}\left(M(f^d+(a f^q)^{\frac{d}{p}}+(bf^s)^{\frac{d}{p}})\right)^{\frac{p}{d}}(z) \,dz\right)^{\frac{\beta}{n+2}}.
\end{align}
Note that for each $z \in E(\Lambda)^c$, there exists a unique $\lambda_z>1$ such that $\Lambda = \lambda^p_z+a(z)\lambda^q_z+b(z)\lambda^s_z$.
We then consider a family of metrics $\{d_z(\cdot, \cdot)\}_{z \in E(\Lambda)^c}$ given by
{ $$\displaystyle
    d_z(z_1,z_2 )=\left\{\begin{array}{l}
    \max\left\{|x_1-x_2|, \sqrt{\lambda^{p-2}_z|t_1-t_2|}\right\}\\ 
    \qquad\qquad\qquad \text{if}\,\,\, K^2\lambda^p_z\geq a(z)\lambda^q_z \,\,\, \text{and}\,\,\,K^2\lambda^p_z\geq b(z)\lambda^s_z,\\
    \max\left\{|x_1-x_2|, \sqrt{g_q(z,\lambda_z)\lambda^{-2}_z|t_1-t_2|}\right\}\\
    \qquad\qquad\qquad \text{if}\,\,\, K^2\lambda^p_z< a(z)\lambda^q_z \,\,\, \text{and}\,\,\,K^2\lambda^p_z\geq b(z)\lambda^s_z,\\
    \max\left\{|x_1-x_2|, \sqrt{g_s(z,\lambda_z)\lambda^{-2}_z|t_1-t_2|}\right\}\\
    \qquad\qquad\qquad \text{if}\,\,\, K^2\lambda^p_z\geq a(z)\lambda^q_z \,\,\, \text{and}\,\,\,K^2\lambda^p_z< b(z)\lambda^s_z,\\
    \max\left\{|x_1-x_2|, \sqrt{g_{q,s}(z,\lambda_z)\lambda^{-2}_z|t_1-t_2|}\right\}\\
    \qquad\qquad\qquad \text{if}\,\,\, K^2\lambda^p_z< a(z)\lambda^q_z \,\,\, \text{and}\,\,\,K^2\lambda^p_z< b(z)\lambda^s_z
    \end{array}\right.
$$
for $z_1=(x_1,t_1),\,z_2=(x_2,t_2)\in \mr^{n+1}$, where the functions \(g_q\), \(g_s\), and \(g_{q,s}\), which are \(q\), \(s\), and \((q,s)\)-growth functions, respectively, defined as:
\begin{align} \label{def: q-growth function}
    g_q(z,\kappa)&:= \kappa^p+a(z)\kappa^q,\\ \label{def: s-growth function}
    g_s(z,\kappa)&:= \kappa^p+b(z)\kappa^s,\\ \label{def: (q,s)-growth function}
    g_{q,s}(z,\kappa)&:= \kappa^p+a(z)\kappa^q+b(z)\kappa^s
\end{align}
for $z\in \Omega_T$ and $\kappa\in \mr^+$.}
For every $z \in E(\Lambda)^c,$ we define a distance function $z$ to $E(\Lambda)$ as
\begin{align}\label{eq : definition of r_z}
    4r_z=d_z(z, E(\Lambda))=\inf_{w \in E(\Lambda)}d_z(z, w).
\end{align}
Using this distance, we construct a family of open subsets $\{U^{\Lambda}_z\}_{z\in E(\Lambda)^c}$ defined as
\begin{align*}
    U^{\Lambda}_z=\begin{cases}
Q_{r_z, \lambda_z} \,\,\quad \text{if}\,\,\, K^2\lambda^p_z\geq a(z)\lambda^q_z \,\,\, \text{and}\,\,\,K^2\lambda^p_z\geq b(z)\lambda^s_z,\\
Q^{q}_{r_z, \lambda_z}\,\,\quad \text{if}\,\,\, K^2\lambda^p_z< a(z)\lambda^q_z \,\,\, \text{and}\,\,\,K^2\lambda^p_z\geq b(z)\lambda^s_z,\\
Q^{s}_{r_z, \lambda_z}\,\,\quad \text{if}\,\,\, K^2\lambda^p_z\geq a(z)\lambda^q_z \,\,\, \text{and}\,\,\,K^2\lambda^p_z< b(z)\lambda^s_z,\\
Q^{q, s}_{r_z, \lambda_z}\,\,\quad \text{if}\,\,\, K^2\lambda^p_z< a(z)\lambda^q_z \,\,\, \text{and}\,\,\,K^2\lambda^p_z< b(z)\lambda^s_z.
    \end{cases}
\end{align*}
Here, the intrinsic cylinders denoted as follows:
\begin{align*}
    Q_{r_z,\lambda_z}(z)&:=B_{r_z}(x)\times I_{r_z,\lambda_z}(t),\\
    Q_{r_z,\lambda_z}^q(z)&:=B_{r_z}(x)\times I_{r_z,\lambda_z}^q(t),\\
    Q_{r_z,\lambda_z}^s(z)&:=B_{r_z}(x)\times I_{r_z,\lambda_z}^s(t),\\
    Q_{r_z,\lambda_z}^{q,s}(z)&:=B_{r_z}(x)\times I_{r_z,\lambda_z}^{q,s}(t)
\end{align*}
for any $z=(x,t)\in \mr^n\times (0,T)$, $r_z>0$, $\lambda_z\geq 1$, where
\begin{align*}
    I_{r_z,\lambda_z}(t)&:=\left(t-\lambda_z^{2-p}r_{z}^2, \ t+\lambda_z^{2-p}r_z^2\right),\\
    I_{r_z,\lambda_z}^q(t)&:=\left(t-\frac{\lambda_z^2}{g_{q}(z,\lambda_z)}r_z^2, \ t+\frac{\lambda_z^2}{g_{q}(z,\lambda_z)}r_z^2\right),\\
    I_{r_z,\lambda_z}^s(t)&:=\left(t-\frac{\lambda_z^2}{g_{s}(z,\lambda_z)}r_z^2, \ t+\frac{\lambda_z^2}{g_{s}(z,\lambda_z)}r_z^2\right),\\
    I_{r_z,\lambda_z}^{q,s}(t)&:=\left(t-\frac{\lambda_z^2}{g_{q,s}(z,\lambda_z)}r_z^2, \ t+\frac{\lambda_z^2}{g_{q,s}(z,\lambda_z)}r_z^2\right).
\end{align*}

Let $f\in L^1(\Omega_T)$ be a function, and let $Q\subset \Omega_T$ be a measurable set with finite positive measure. We define the integral average of $f$ over $Q$ by $$(f)_Q:=\miint{Q} f \,dz.$$
If $Q$ is one of the four intrinsic cylinders defined above, we denote it as follows:
\begin{align*}
    (f)_{z_0;\rho,\lambda}&:=\miint{Q_{\rho,\lambda}(z_0)}f\,dz,\\
    (f)_{z_0;\rho,\lambda}^{(q)}&:=\miint{Q_{\rho,\lambda}^q(z_0)}f\,dz,\\
    (f)_{z_0;\rho,\lambda}^{(s)}&:=\miint{Q_{\rho,\lambda}^s(z_0)}f\,dz,\\
    (f)_{z_0;\rho,\lambda}^{(q,s)}&:=\miint{Q_{\rho,\lambda}^{q,s}(z_0)}f\,dz.
\end{align*}
The integral average on a ball $B_\rho(x_0)\subset \Omega$ is denoted by
$$
(f)_{x_0;\rho}(t):= \dashint_{B_\rho(x_0)}f(x,t)\, dx, \quad t\in(0,T).
$$
{ We intend to use a Vitali-type argument to find a countable collection of points $z_i \in E(\Lambda)^c$ such that the corresponding subfamily satisfies some properties.} From now on, we denote
$$
\begin{array}{c}
    \lambda_i=\lambda_{z_i},\,\,\,\, d_i(\cdot, \cdot)=d_{z_i}(\cdot, \cdot)\\
    Q_i=Q_{r_{z_i},\lambda_{z_i}},\quad  Q^{q}_{i}=Q^q_{r_{z_i}, \lambda_{z_i}},\quad Q^{s}_{i}=Q^s_{r_{z_i}, \lambda_{z_i}},\quad Q^{q, s}_{i}=Q^{q,s}_{r_{z_i}, \lambda_{z_i}}
\end{array}
$$
and 
\begin{align} \label{EQQ2.4}
U_i=B_i\times I_i=\begin{cases}
    Q_i \,\,\, \quad \text{if}\,\,\, K^2\lambda^p_i\geq a(z_i)\lambda^q_i \,\,\, \text{and}\,\,\,K^2\lambda^p_i\geq b(z_i)\lambda^s_i,\\
Q^{q}_i\,\,\, \quad \text{if}\,\,\, K^2\lambda^p_i< a(z_i)\lambda^q_i \,\,\, \text{and}\,\,\,K^2\lambda^p_i\geq b(z_i)\lambda^s_i,\\
Q^{s}_i\,\,\, \quad \text{if}\,\,\, K^2\lambda^p_i\geq a(z_i)\lambda^q_i \,\,\, \text{and}\,\,\,K^2\lambda^p_i< b(z_i)\lambda^s_i,\\
Q^{q, s}_i\,\,\,\ \text{if}\,\,\, K^2\lambda^p_i< a(z_i)\lambda^q_i \,\,\, \text{and}\,\,\,K^2\lambda^p_i< b(z_i)\lambda^s_i.
\end{cases}
\end{align}
Moreover, we denote
 \begin{equation}\label{EQQ2.5}
    \begin{aligned}
    &d_i(U_i, E(\Lambda))=\inf_{z\in U_i, w\in E(\Lambda)}d_i(z, w),\quad \mathcal{I}=\left\{j\in \mathbb{N}: \frac{2}{K}U_i\cap \frac{2}{K}U_j\neq \emptyset\right\},\\
    &K_i=200K^3 \,\,\, \text{for}\,\,\, U_i=Q_i,\,\,\, Q^q_i,\,\,\, \text{or}\,\,\, Q^s_i\,\,\, \text{and}\,\,\, K_i=200 \,\,\, \text{for}\,\,\, U_i=Q^{q,s}_i.
\end{aligned}
\end{equation}

\subsection{\bf Some preliminary estimates} We start with some basic estimates.
\begin{lemma}\label{LEM2.1}
    Let $z\in E(\Lambda)^c$. Assume $K^2 \lambda^p_z< a(z)\lambda^q_z$ and $K^2 \lambda^p_z< b(z)\lambda^s_z.$ Then $\frac{a(z)}{2}\leq a(\tilde{z})\leq 2a(z)$ and $\frac{b(z)}{2}\leq b(\tilde{z})\leq 2b(z)$ for any $\tilde{z} \in 200K Q_{r_z}(z)$.
\end{lemma}
\begin{proof}
    We claim that $2[a]_{\alpha}(200 K r_z)^{\alpha}< a(z)$ and $2[b]_{\beta}(200 K r_z)^{\beta}< b(z).$ We will only prove the second statement, and the proof of the first statement is similar. On the contrary, let us assume $b(z)\leq 2[b]_{\beta}(200 K r_z)^{\beta}.$ Since $Q^{q, s}_{r_z, \lambda_z}(z)\subset E(\Lambda)^c$ and $a(z)\lambda^q_z+b(z)\lambda^s_z < \Lambda$, we get
    \begin{align*}
    a(z)\lambda^q_z+b(z)\lambda^s_z< \miint{Q^{q, s}_{r_z, \lambda_z}(z)}\left(M \left(f^d+(a f^q)^{d/p}+(b f^s)^{d/p}\right)(w)\right)^{p/d}\,dw.
    \end{align*}
    Now using $\Lambda = \lambda^p_z+a(z)\lambda^q_z+b(z)\lambda^s_z < 2[a(z)\lambda^q_z+b(z)\lambda^s_z]$, we obtain
 \begin{align*}
 &a(z)\lambda^q_z+b(z)\lambda^s_z< \miint{Q^{q, s}_{r_z, \lambda_z}(z)}\left(M \left(f^d+(a f^q)^{d/p}+(b f^s)^{d/p}\right)(w)\right)^{p/d}\,dw\\
 &\qquad\qquad= \frac{\Lambda \lambda^{-2}_z}{2|B_1|r^{n+2}_z}\iint_{Q^{q, s}_{r_z, \lambda_z}(z)}\left(M \left(f^d+(a f^q)^{d/p}+(b f^s)^{d/p}\right)(w)\right)^{p/d}\,dw\\
 &\qquad\qquad<\frac{(a(z)\lambda^q_z+b(z)\lambda^s_z)\lambda^{-2}_z}{|B_1|r^{n+2}_z}\\
 &\qquad\qquad\qquad\times\iint_{Q^{q, s}_{r_z, \lambda_z}(z)}\left(M \left(f^d+(a f^q)^{d/p}+(b f^s)^{d/p}\right)(w)\right)^{p/d}\,dw.
    \end{align*}
    Raising the power to $\frac{\beta}{n+2}$ in both of sides to the above expression, we have
 \begin{align*}
        r^{\beta}_z\lambda^{\frac{2\beta}{n+2}}_z&<\left(\frac{1}{|B_1|}\iint_{Q^{q, s}_{r_z, \lambda_z}(z)}\left(M \left(f^d+(a f^q)^{d/p}+(b f^s)^{d/p}\right)(w)\right)^{p/d}\,dw\right)^{\frac{\beta}{n+2}}\\
        &\leq \frac{K}{800 [b]_{\beta}}
    \end{align*}
    Now using $K^2 \lambda^p_z< b(z)\lambda^s_z,$ $s \leq p+\frac{2\beta}{n+2}$ and the counter assumption, we get
    \begin{align*}
        K^2 \lambda^p_z< b(z)\lambda^s_z\leq 2[b]_{\beta}(200 K r_z)^{\beta}\lambda^p_z\lambda^{\frac{2\beta}{n+2}}_z\leq \frac{1}{2}K^2\lambda^p_z,
    \end{align*}
    which is a contradiction. This completes the proof of $2[b]_{\beta}(200 K r_z)^{\beta}< b(z).$ Now using the H\"{o}lder continuity of $b(z),$ we get
    \begin{align*}
        2[b]_{\beta}(200 K r_z)^{\beta}< b(z) \leq \inf_{\tilde{z} \in 200K Q_{r_z}(z)} b(\tilde{z})+ [b]_{\beta}(200K r_z)^{\beta},
    \end{align*}
    and hence
    \begin{align*}
      [b]_{\beta}(200 K r_z)^{\beta} \leq \inf_{\tilde{z} \in 200K Q_{r_z}(z)} b(\tilde{z}).  
    \end{align*}
It follows that
\begin{align*}
    \sup_{\tilde{z}\in 200K Q_{r_z}(z)}b(\tilde{z}) \leq \inf_{\tilde{z} \in 200K Q_{r_z}(z)} b(\tilde{z})+[b]_{\beta}(200 K r_z)^{\beta}<2\inf_{\tilde{z} \in 200K Q_{r_z}(z)}b(\tilde{z})
\end{align*}
and this proves $\frac{b(z)}{2}\leq b(\tilde{z}) \leq 2b(z)$ for any $\tilde{z} \in 200K Q_{r_z}(z).$ The statement for $a(\cdot)$ can be proved similarly.
\end{proof}
{ \begin{lemma}   \label{lem : comparision of b(cdot) in p,s-phase}
 Let $z\in E(\Lambda)^c$. Assume $K^2 \lambda^p_z\geq a(z)\lambda^q_z$ and $K^2 \lambda^p_z< b(z)\lambda^s_z.$ Then $\frac{b(z)}{2}\leq b(\tilde{z})\leq 2b(z)$ for any $\tilde{z} \in 200K Q_{r_z}(z).$ Moreover, we have 
 \begin{equation}\label{EQ2.1}
 [a]_{\alpha}(800 K r_z)^{\alpha}\lambda^q_z\leq (K^2-1)\lambda^p_z.
 \end{equation}
\end{lemma}
\begin{proof}
    First, we note that $Q^{s}_{r_z, \lambda_z}(z)\subset E(\Lambda)^c,$ and therefore
    \begin{align*}
        &g_s(z,\lambda_z) < \miint{Q^{s}_{r_z, \lambda_z}(z)}\left(M \left(f^d+(a f^q)^{d/p}+(b f^s)^{d/p}\right)(w)\right)^{p/d}\,dw\\
        &\qquad=\frac{g_s(z,\lambda_z)\lambda^{-2}_z}{2|B_1|r^{n+2}_z}\iint_{Q^{s}_{r_z, \lambda_z}(z)}\left(M \left(f^d+(a f^q)^{d/p}+(b f^s)^{d/p}\right)(w)\right)^{p/d}\,dw.
    \end{align*}
    Raising power $\frac{\alpha}{n+2}$ in the above expression, we have
    \begin{align*}
     r^{\alpha}_z \lambda^{\frac{2\alpha}{n+2}}_z&< \left(\frac{1}{|B_1|}\iint_{Q^{s}_{r_z, \lambda_z}(z)}  \left(M \left(f^d+(a f^q)^{d/p}+(b f^s)^{d/p}\right)(w)\right)^{p/d}\,dw\right)^{\frac{\alpha}{n+2}}\\ 
     &\leq \frac{K-1}{800[a]_{\alpha}}. 
    \end{align*}
Using $q\leq p+\frac{2\alpha}{n+2},$ we get
\begin{align*}
    [a]_{\alpha}(800K r_z)^{\alpha}\lambda^q_z \leq [a]_{\alpha}800K r^{\alpha}_z\lambda^p_z \lambda^{\frac{2\alpha}{n+2}}_z &\leq [a]_{\alpha}800K \lambda^p_z\frac{K-1}{800[a]_{\alpha}}\\
    &\leq (K^2-1)\lambda^p_z.
\end{align*}
    
    Moreover, the proof of the first statement follows from the previous lemma. 
This completes the proof.
\end{proof}
\begin{lemma}   \label{lem : comparision of a(cdot) in p,q-phase}
    Let $z\in E(\Lambda)^c$. Assume $K^2 \lambda^p_z< a(z)\lambda^q_z$ and $K^2 \lambda^p_z\geq b(z)\lambda^s_z.$ Then $\frac{a(z)}{2}\leq a(\tilde{z})\leq 2a(z)$ for any $\tilde{z} \in 200K Q_{r_z}(z).$ Furthermore, we have \begin{equation}\label{Equation2.14} 
    [b]_{\beta}(800 K r_z)^{\beta}\lambda^s_z\leq (K^2-1)\lambda^p_z.\end{equation}
\end{lemma}
\begin{proof}
Since $Q^q_{r_z, \lambda_z} \subset E(\Lambda)^c,$ we have
 \begin{align*}
    &g_q(z,\lambda_z)< \miint{Q^{s}_{r_z, \lambda_z}(z)}\left(M \left(f^d+(a f^q)^{d/p}+(b f^s)^{d/p}\right)(w)\right)^{p/d}\,dw\\
        &\qquad=\frac{g_{q}(z,\lambda_z)\lambda^{-2}_z}{2|B_1|r^{n+2}_z}\iint_{Q^{s}_{r_z, \lambda_z}(z)}\left(M \left(f^d+(a f^q)^{d/p}+(b f^s)^{d/p}\right)(w)\right)^{p/d}\,dw.
\end{align*}
Moreover, following the previous lemma, we get
\begin{align*}
    [b]_{\beta}(800 K r_z)^{\beta}\lambda^s_z \leq (K^2-1)\lambda^p_z.
\end{align*}
Also, the first statement follows from Lemma \ref{LEM2.1}.
This completes the proof.
\end{proof}
\begin{lemma}   \label{lem : relation of lambda z and lambda tilde z in p,s-phase}
    Let $z, \tilde{z} \in E(\Lambda)^c$.  Assume $K^2 \lambda^p_z \geq a(z)\lambda^q_z$ and $K^2\lambda^p_z <b(z)\lambda^s_z,$ and $\tilde{z} \in 200K Q_{r_z}(z)$. Then $K^{-\frac{2}{p}}\lambda_{\tilde{z}} \leq \lambda_{z} \leq K^{\frac{2}{p}}\lambda_{\tilde{z}}$.
\end{lemma}
\begin{proof}
    From Lemma \ref{LEM2.1}, we get
\begin{align}\label{EQ2.2}
  \frac{b(z)}{2}\leq b(\tilde{z})\leq 2b(z).  
\end{align} 
Now, note that it is enough to prove $K^{-\frac{2}{p}}\lambda_{\tilde{z}} \leq \lambda_z.$ On contrary, let us assume $\lambda_z < K^{-\frac{2}{p}}\lambda_{\tilde{z}}.$ Then using \eqref{EQ2.1}, \eqref{EQ2.2} and the counter assumption, we obtain
\begin{align*}
    \Lambda&=\lambda_z^p+a(z)\lambda_z^q+b(z)\lambda_z^s\\
    &\leq \lambda^p_z+[a]_{\alpha}(200K r_z)^{\alpha}\lambda^q_z+a(\tilde{z})\lambda^q_z+b(z)\lambda^s_z\\
    &< \lambda^p_z+(K^2-1)\lambda^p_z+ a(\tilde{z})K^{-\frac{2q}{p}}\lambda^q_{\tilde{z}}+2b(\tilde{z})K^{-\frac{2s}{p}}\lambda^s_{\tilde{z}}\\
    &\leq K^2\lambda^p_z+ a(\tilde{z})K^{-\frac{2q}{p}}\lambda^q_{\tilde{z}}+ b(\tilde{z})K^{1-\frac{2s}{p}}\lambda^s_{\tilde{z}}\\
    &< \lambda_{\tilde{z}}^p+a(\tilde{z})\lambda_{\tilde{z}}^q+b(\tilde{z})\lambda_{\tilde{z}}^s=\Lambda,
\end{align*}
which is a contradiction. 
\end{proof}
\begin{lemma}   \label{lem : relation of lambda z and lambda tilde z in p,q-phase}
    Let $z, \tilde{z} \in E(\Lambda)^c$. Assume $K^2 \lambda^p_z < a(z)\lambda^q_z$ and $K^2\lambda^p_z\geq b(z)\lambda^s_z,$ and $\tilde{z} \in 200K Q_{r_z}(z).$  Then $K^{-\frac{2}{p}}\lambda_{\tilde{z}} \leq \lambda_z \leq K^{\frac{2}{p}}\lambda_{\tilde{z}}$. 
\end{lemma}
\begin{proof}
From Lemma \ref{lem : comparision of a(cdot) in p,q-phase}, we have
\begin{align}\label{Equation2.15}
    \frac{a(z)}{2}\leq a(\tilde{z})\leq a(z)
\end{align}
for any $\tilde{z}\in 200KQ_{r_z}(z).$
We will show $K^{-\frac{2}{p}}\leq \lambda_{\tilde{z}}.$ On contrary, let us assume $K^{-\frac{2}{p}}> \lambda_{\tilde{z}}.$ Then using \eqref{Equation2.14} and \eqref{Equation2.15}, we get
\begin{align*}
    \Lambda&=\lambda_z^p+a(z)\lambda_z^q+b(z)\lambda_z^s\\
    &\leq \lambda^p_z+a(z)\lambda^q_z+[b]_{\beta}(200K r_z)^{\beta}\lambda^s_z+b(\tilde{z})\lambda^s_z\\
    &< \lambda^p_z+(K^2-1)\lambda^p_z+ 2a(\tilde{z})K^{-\frac{2q}{p}}\lambda^q_{\tilde{z}}+b(\tilde{z})K^{-\frac{2s}{p}}\lambda^s_{\tilde{z}}\\
    &\leq K^2\lambda^p_z+ a(\tilde{z})K^{1-\frac{2q}{p}}\lambda^q_{\tilde{z}}+ b(\tilde{z})K^{-\frac{2s}{p}}\lambda^s_{\tilde{z}}\\
    &\leq \lambda_{\tilde{z}}^p+a(\tilde{z})\lambda_{\tilde{z}}^q+b(\tilde{z})\lambda_{\tilde{z}}^s=\Lambda,
\end{align*}
which gives a contradiction. This completes the proof.
\end{proof}}
\begin{lemma}   \label{lem : relation of lambda z and lambda tilde z in p,q,s-phase}
Let $z, \tilde{z} \in E(\Lambda)^c$. Assume $\frac{a(z)}{2}\leq a(\tilde{z})\leq 2a(z)$  and $\frac{b(z)}{2}\leq b(\tilde{z})\leq 2b(z).$ Then $2^{-\frac{1}{p}}\lambda_{\tilde{z}} \leq \lambda_z \leq 2^{\frac{1}{p}}\lambda_{\tilde{z}}.$ Moreover, the above inequality holds provided that $K^2 \lambda^p_z< a(z)\lambda^q_z$ and $K^2\lambda^p_z <b(z)\lambda^s_z,$ and $\tilde{z} \in 200K Q_{r_z}(z).$    
\end{lemma}
\begin{proof}
    It is enough to prove the first statement. The second statement of the lemma follows from Lemma \ref{LEM2.1}. We claim $\lambda_z\leq 2^{1/p}\lambda_{\tilde{z}}.$ In contrast, assume $\lambda_z > 2^{1/p}\lambda_{\tilde{z}}.$ Using the hypothesis and the counter assumption, we get
    \begin{align*}
        \Lambda=\lambda_z^p+a(z)\lambda_z^q+b(z)\lambda_z^s&> 2 \lambda^p_{\tilde{z}}+2^{q/p-1}a(\tilde{z})\lambda^q_{\tilde{z}}+2^{s/p-1}b(\tilde{z})\lambda^s_{\tilde{z}}\\
        &\geq \lambda_{\tilde{z}}^p+a(\tilde{z})\lambda_{\tilde{z}}^q+b(\tilde{z})\lambda_{\tilde{z}}^s=\Lambda,
    \end{align*}
    which is a contradiction.
\end{proof}
\begin{lemma}
    Let $z \in E(\Lambda)^c$. Assume that $K^2 \lambda^p_z\geq a(z)\lambda^q_z$ and $K^2 \lambda^p_z \geq b(z)\lambda^s_z.$ Then $[a]_{\alpha}(50K r_z)^{\alpha}\lambda^q_z\leq (K^2-1)\lambda^p_z$ and $[b]_{\beta}(50K r_z)^{\beta}\lambda^s_z\leq (K^2-1)\lambda^p_z$.
\end{lemma}
\begin{proof}
    The proof can be completed analogously to the argument presented in the proof of Lemma \ref{LEM2.1}.
\end{proof}
\begin{lemma}   \label{lem : relation of lambda z and lambda tilde z in p-phase}
Let $z \in E(\Lambda)^c$. Assume that $K^2 \lambda^p_z\geq a(z)\lambda^q_z$ and $K^2 \lambda^p_z \geq b(z)\lambda^s_z.$ If $\tilde{z} \in 50KQ_{r_z}(z),$ then $\lambda_z\leq K^{2/p}\lambda_{\tilde{z}}.$    
\end{lemma}
\begin{proof}
    The proof can be completed from Lemma \ref{lem : relation of lambda z and lambda tilde z in p,s-phase} and Lemma \ref{lem : relation of lambda z and lambda tilde z in p,q-phase}.
\end{proof}
\subsection{Vitali covering and their properties}   \label{subsection : Vitali covering}
In this subsection, we want to choose a countable collection of intrinsic cylinders from $\{U_z^\Lambda\}_{z\in E(\Lambda)^c}$ such that 
$$
E(\Lambda)^c \subset \bigcup_{z\in E(\Lambda)^c} \frac{1}{K}U_z^\lambda\quad \text{and}\quad \frac{1}{6K^6}U_{z_1}^\Lambda \cap \frac{1}{6K^6}U_{z_2}^\Lambda=\emptyset \quad \text{for any }z_1,\, z_2\in E(\Lambda)^c.
$$
First, we claim that $\{r_z: z\in E(\Lambda)^c\}$ is uniformly bounded. Indeed, since, by \eqref{limit of Λ|E(Λ)^c|}, $0<|E(\Lambda)^c|<\infty$ and $\Lambda>1+\|a\|_{L^\infty (\mr^{n+1})}+\|b\|_{L^\infty(\mr^{n+1})}$, there exist $\lambda>1$ and $R>0$, which are in particular independent of $z\in E(\Lambda)^c$, such that 
\begin{equation}    \label{eq : def of lambda}
    \lambda^p + \|a\|_{L^\infty (\mr^{n+1})}\lambda^q+\|b\|_{L^\infty (\mr^{n+1})}\lambda^s=\Lambda
\end{equation}
and $|B_R\times (-\lambda^2\Lambda^{-1}R^2,\lambda^2\Lambda^{-1}R^2)|=|E(\Lambda)^c|$. It clear that $\lambda\leq \lambda_z\leq \Lambda$ for any $z\in E(\Lambda)^c$. Hence, if $r_z> R$ for some $z\in E(\Lambda)^c$, then 
$$
|E(\Lambda)^c|=|B_R\times (-\lambda^2\Lambda^{-1}R^2,\lambda^2\Lambda^{-1}R^2)|<|U_z^\Lambda|\leq |E(\Lambda)^c|.
$$
This is a contradiction, and we conclude $r_z\leq R$ for any $z\in E(\Lambda)^c$.

Let $\mathcal{F}=\left\{\frac{1}{6K^6}U_z^\Lambda\right\}_{z\in E(\Lambda)^c}$ and, for each $j\in\mathbb{N}$,
$$
\mathcal{F}_j=\left\{\frac{1}{6K^6}U_z^\Lambda\in \mathcal{F}:\frac{R}{2^j}<r_z\leq \frac{R}{2^{j-1}}\right\}.
$$
Since $r_z\leq R$ for any $z\in E(\Lambda)^c$, we obtain 
$$
\mathcal{F}=\bigcup_{j\in \mathbb{N}} \mathcal{F}_j.
$$
Then, in the same way as \cite[Subsection 3.3]{Wontae2023a}, we have a countable subcollection $\mathcal{G}$ of pairwise disjoint cylinders in $\mathcal{F}$. From now on, we show that the $6K^5$-times the cylinders in $\mathcal{G}$ covers $E(\Lambda)^c$. Fix $\frac{1}{6K^6}U_{\tilde{z}}^\Lambda\in \mathcal{F}$. Then there exists $i\in\mathbb{N}$ such that $\frac{1}{6K^6}U_{\tilde{z}}^\Lambda\in \mathcal{F}_i$ and by the construction of $\mathcal{G}$ in \cite{Wontae2023a} there exists a cylinder $\frac{1}{6K^6}U_z^\Lambda\in \cup_{j=1}^i \mathcal{G}_j$ such that 
\begin{equation}    \label{eq : 1/6KU_tilde z intersect 1/6KU_z}
\frac{1}{6K^6}U_{\tilde{z}}^\Lambda\cap \frac{1}{6K^6}U_z^\Lambda \neq \emptyset.
\end{equation}
Since $\frac{1}{6K^6}U_{\tilde{z}}^\Lambda\in \mathcal{F}_i$ and $\frac{1}{6K^6}U_z^\Lambda\in \mathcal{F}_j$ for some $j\in\{1,\cdots,i\}$, we obtain
\begin{equation}    \label{eq : relation of r_ tilde z and r_z}
r_{\tilde{z}}\leq 2r_z.
\end{equation}

We claim that 
\begin{equation}    \label{eq : covering in R n+1}
\frac{1}{6K^6} U_{\tilde{z}}^\Lambda\subset \frac{1}{K}U_z^\Lambda.
\end{equation}
Let $z=(x,t)$ and $\tilde{z}=(\tilde{x},\tilde{t})$. By \eqref{eq : 1/6KU_tilde z intersect 1/6KU_z}, \eqref{eq : relation of r_ tilde z and r_z} and the proof of the standard Vitali covering lemma, we obtain 
$$
\frac{1}{6K^6}B_{r_{\tilde{z}}}(y)\subset \frac{1}{K}B_{r_z}(x).
$$
Moreover, by \eqref{eq : relation of r_ tilde z and r_z} and the standard Vitali covering argument, we obtain that $Q_{r_{\tilde{z}}}({\tilde{z}})\subset 5Q_{r_z}(z)$ and hence, we conclude 
\begin{equation}    \label{eq : tilde z belong in 50KQ}
\tilde{z}\in 5Q_{r_z}(z)\subset 200K Q_{r_z}(z).
\end{equation}
Now, we prove the inclusion in \eqref{eq : covering in R n+1} in the time direction, by considering the $16$ cases depicted in Table \ref{tab : 16 Cases}. 
\begin{table}[t]
    \centering
    \begin{tabular}{|c|c c c c|}
        \hline
        \backslashbox{$U_{\tilde{z}}^\Lambda$}{$U_z^\Lambda$} & $Q_{r_z,\lambda_z}(z)$ & $Q_{r_z,\lambda_z}^q(z)$ & $Q_{r_z,\lambda_z}^s(z)$ & $Q_{r_z,\lambda_z}^{q,s}(z)$\\ 
        \hline
        $Q_{r_{\tilde{z}},\lambda_{\tilde{z}}}(\tilde{z})$ & (1-1) & (1-2) & (1-3) & (1-4) \\
        
        $Q_{r_{\tilde{z}},\lambda_{\tilde{z}}}^q(\tilde{z})$ & (2-1) & (2-2) & (2-3) & (2-4) \\
        
        $Q_{r_{\tilde{z}},\lambda_{\tilde{z}}}^s(\tilde{z})$ & (3-1) & (3-2) & (3-3) & (3-4) \\
        
        $Q_{r_{\tilde{z}},\lambda_{\tilde{z}}}^{q,s}(\tilde{z})$ & (4-1) & (4-2) & (4-3) & (4-4) \\
        \hline
    \end{tabular}
    \caption{The combinations of $U_{\tilde{z}}^\Lambda$ and $U_z^\Lambda$.}
    \label{tab : 16 Cases} 
\end{table}

\textit{Case $(1$-$1)$}. By \eqref{eq : 1/6KU_tilde z intersect 1/6KU_z} and \eqref{eq : relation of r_ tilde z and r_z}, for $\tau\in\frac{1}{6K^6}I_{r_{\tilde{z}},\lambda_{\tilde{z}}}(\tilde{t})$, we obtain 
$$
\begin{aligned}
    |\tau - t|&\leq |\tau-\tilde{t}|+|\tilde{t}-t|\\
    &\leq 2\lambda_{\tilde{z}}^{2-p}\left(\frac{r_{\tilde{z}}}{6K^6}\right)^2+\lambda_z^{2-p}\left(\frac{r_z}{6K^6}\right)^2\\
    &\leq (8\lambda_{\tilde{z}}^{2-p}+\lambda_z^{2-p})\left(\frac{r_z}{6K^6}\right)^2.
\end{aligned}
$$
Lemma \ref{lem : relation of lambda z and lambda tilde z in p-phase} and \eqref{eq : tilde z belong in 50KQ} imply
$$
\begin{aligned}
    |\tau - t|&\leq (8K^{\frac{2(p-2)}{p}}+1)\lambda_z^{2-p}\left(\frac{r_z}{6K^6}\right)^2\\
    &\leq (9K^2)\lambda_z^{2-p}\left(\frac{r_z}{6K^6}\right)^2\\
    &\leq \lambda_z^{2-p}\left({\frac{r_z}{K}}\right)^2,
\end{aligned}
$$
and hence $\tau\in \frac{1}{K}I_{r_z,\lambda_z}(t)$.

\textit{Cases $(2$-$1)$, $(3$-$1)$ and $(4$-$1)$.} Since $Q_{r_{\tilde{z}},\lambda_{\tilde{z}}}^q({\tilde{z}})\subset Q_{r_{\tilde{z}},\lambda_{\tilde{z}}}({\tilde{z}})$, $Q_{r_{\tilde{z}},\lambda_{\tilde{z}}}^s({\tilde{z}})\subset Q_{r_{\tilde{z}},\lambda_{\tilde{z}}}({\tilde{z}})$ and $Q_{r_{\tilde{z}},\lambda_{\tilde{z}}}^{q,s}({\tilde{z}})\subset Q_{r_{\tilde{z}},\lambda_{\tilde{z}}}({\tilde{z}})$, by the previous argument, we obtain the conclusion.

\textit{Case $(1$-$2)$} By \eqref{eq : 1/6KU_tilde z intersect 1/6KU_z} and \eqref{eq : relation of r_ tilde z and r_z}, for $\tau\in \frac{1}{6K^6}I_{r_{\tilde{z}},\lambda_{\tilde{z}}}(\tilde{t})$, we obtain 
$$
\begin{aligned}
    |\tau-t|&\leq |\tau-\tilde{t}|+|\tilde{t}-t|\\
    &\leq 2\lambda_{\tilde{z}}^{2-p}\left(\frac{r_{\tilde{z}}}{6K^6}\right)^2+\frac{\lambda^2_z}{g_q(z,\lambda_z)}\left(\frac{r_z}{6K^6}\right)^2\\
    &\leq \left(8\lambda_{\tilde{z}}^{2-p}+\frac{\lambda^2_z}{g_q(z,\lambda_z)}\right)\left(\frac{r_z}{6K^6}\right)^2.
\end{aligned} 
$$
From $a(\tilde{z})\lambda_{\tilde{z}}^q\leq K^2\lambda_{\tilde{z}}^p$, we have 
$$
\lambda_{\tilde{z}}^{2-p}=2\frac{\lambda_{\tilde{z}}^2}{\lambda_{\tilde{z}}^p+\lambda_{\tilde{z}}^p}\leq 2K^2\frac{\lambda_{\tilde{z}}^2}{g_q(\tilde{z},\lambda_{\tilde{z}})}.
$$
Since \eqref{eq : condition of n,p,q,alpha} implies $\frac{q}{p}\leq 2$, we obtain from Lemma \ref{lem : comparision of a(cdot) in p,q-phase}, Lemma \ref{lem : relation of lambda z and lambda tilde z in p,q-phase}, \eqref{eq : condition of n,p,q,alpha} and \eqref{eq : tilde z belong in 50KQ} that
\begin{equation}    \label{eq : equation in Case (1-2)}
    \frac{\lambda_{\tilde{z}}^2}{g_q(\tilde{z},\lambda_{\tilde{z}})}\leq \frac{K^6\lambda_z^2}{\lambda_z^p+a(\tilde{z})\lambda_z^q}\leq 2K^6\frac{\lambda_z^2}{g(z,\lambda_z)}.
\end{equation}
Thus, we conclude
$$
|\tau-t|\leq \left(32K^8+1\right)\frac{\lambda^2_z}{g_q(z,\lambda_z)}\left(\frac{r_z}{6K^6}\right)^2\leq \frac{\lambda^2_z}{g_q(z,\lambda_z)}\left(\frac{r_z}{K}\right)^2
$$
and hence $\tau \in \frac{1}{K}I^q_{r_z,\lambda_z}(t)$.

\textit{Cases $(1$-$3)$, $(1$-$4)$, $(2$-$3)$, $(3$-$2)$, $(2$-$4)$ and $(3$-$4)$}. By Lemmas \ref{lem : comparision of a(cdot) in p,q-phase}, \ref{lem : comparision of b(cdot) in p,s-phase}, \ref{lem : relation of lambda z and lambda tilde z in p,s-phase}, \ref{lem : relation of lambda z and lambda tilde z in p,q-phase} and \ref{lem : relation of lambda z and lambda tilde z in p,q,s-phase}, we obtain from the above argument the conclusion.

\textit{Case $(2$-$2)$} By \eqref{eq : 1/6KU_tilde z intersect 1/6KU_z} and \eqref{eq : relation of r_ tilde z and r_z}, for $\tau \in \frac{1}{6K^6}$, we obtain
$$
\begin{aligned}
    |\tau-t|&\leq |\tau-\tilde{t}|+|\tilde{t}-t|\\
    &\leq 2\frac{\lambda_{\tilde{z}}^2}{g_q(\tilde{z},\lambda_{\tilde{z}})}\left(\frac{r_{\tilde{z}}}{6K^6}\right)^2+\frac{\lambda^2_z}{g_q(z,\lambda_z)}\left(\frac{r_z}{6K^6}\right)^2\\
    &\leq \left(\frac{8\lambda_{\tilde{z}}^2}{g_q(\tilde{z},\lambda_{\tilde{z}})}+\frac{\lambda^2_z}{g_q(z,\lambda_z)}\right)\left(\frac{r_z}{6K^6}\right)^2.
\end{aligned} 
$$
Then \eqref{eq : equation in Case (1-2)} implies 
$$
|\tau-t|\leq \left(16K^8+1\right)\frac{\lambda^2_z}{g_q(z,\lambda_z)}\left(\frac{r_z}{6K^6}\right)^2\leq \frac{\lambda_z^2}{g_q(z,\lambda_z)}\left(\frac{r_z}{K}\right)^2
$$
and hence $\tau \in \frac{1}{K}I^q_{r_z,\lambda_z}(t)$.

\textit{Cases $(3$-$3)$ and $(4$-$4)$.}
In the same way as the above case, we easily obtain the conclusion.

\textit{Cases $(4$-$2)$ and $(4$-$3)$.}
We obtain the conclusion for Cases (2-2) and (2-3) in the same way for Cases (2-1), (3-1) and (4-1).

Thus, we have established a countable covering family $\{\frac{1}{K}U_i\}_{i\in\mathbb{N}}$ of intrinsic cylinders defined as in \eqref{EQQ2.4} and with pairwise disjoint $\frac{1}{6K^6} U_i$. Now, we prove some properties of the collection $\{U_i\}_{i\in\mathbb{N}}$ that will be summarized in Lemma \ref{LEM2.9} at the end.
\begin{lemma}   \label{lem : property of distance between U_i and E(Lambda)}
    We have $3r_i\leq d_i(U_i,E(\Lambda))\leq 4r_i$ for every $i\in \mathbb{N}$.
\end{lemma}
\begin{proof}
    Since, by the definition of $r_i$ in \eqref{eq : definition of r_z},
    $$
    d_i (U_i,E(\Lambda))\leq d_i(z_i,E(\Lambda))=4r_i,
    $$
    the second inequality is satisfied. Moreover, from the triangle inequality we obtain that, for any $z\in U_i$,
    $$
    d_i(z,E(\Lambda))\geq d_i(z_i,E(\Lambda))-d_i(z,z_i)\geq 4r_i-r_i=3r_i.
    $$
    Since $z\in U_i$ is arbitrary, we have $d_i(U_i,E(\Lambda))\geq 3r_i$.
\end{proof}
 Next, we prove the property \ref{v} in Lemma \ref{LEM2.9}.
\begin{lemma}   \label{lem : comparision of r_j and r_i}
    We have $(12K^2)^{-1}r_j\leq r_i\leq 12K^2r_j$ for every $i\in \mathbb{N}$ and $j\in\mathcal{I}$. Moreover, if $U_j=Q^{q,s}_j$, then $r_i\leq 12 r_j$   
\end{lemma}
\begin{proof}
    It is enough to prove that 
    $$
    r_i\leq 12K^2 r_j
    $$
    for $j\in\mathbb{N}$. If $r_i<r_j$, then clearly the conclusion is satisfied. Thus, we assume $r_j\leq r_i$. Then 
    \begin{equation}\label{eq : z_j in cylinder of z_i}
        z_j\in 4Q_{r_i}(z_i).
    \end{equation}
    Let $w\in \frac{2}{K}U_i \cap \frac{2}{K}U_j$. Since $d_i(z_i,w)\leq \frac{2}{K}r_i\leq 2r_i$ and $d_j(z_j,w)\leq \frac{2}{K}r_j\leq 2r_j$, using Lemma \ref{lem : property of distance between U_i and E(Lambda)} and triangle inequality as in \cite[Lemma 3.8]{Wontae2023a}, we get
    \begin{equation}    \label{eq : estimate of d i and d j}
        r_i\leq d_i(w,E(\Lambda)) \quad \text{and}\quad d_j(w,E(\Lambda))\leq 6r_j.
    \end{equation}
    To complete the proof, we consider $16$ cases in Table \ref{tab : 16 Cases} with $z=z_i$ and $\tilde{z}=z_j$.

    \textit{Case $(1$-$1)$, $(2$-$1)$, $(3$-$1)$, $(4$-$1)$, $(1$-$2)$, $(1$-$3)$ and $(1$-$4)$}. Using Lemma \ref{lem : relation of lambda z and lambda tilde z in p,s-phase}, Lemma \ref{lem : relation of lambda z and lambda tilde z in p,q-phase} and Lemma \ref{lem : relation of lambda z and lambda tilde z in p,q,s-phase} and following the proof of Case 1, Case 2 and Case 3 in \cite[Lemma 3.8]{Wontae2023a}, we have the conclusion.

    \textit{Case $(2$-$2)$}. We obtain from Lemma \ref{lem : relation of lambda z and lambda tilde z in p,q-phase} and $K^2\lambda_j^p\geq b(z_j)\lambda_j^s$ that 
    {$$
    \begin{aligned}
        g_q(z_i,\lambda_i)\lambda_i^{-2}&\leq g_{q,s}(z_i,\lambda_i)\lambda_i^{-2}\\
        &\leq K^{\frac{4}{p}}g_{q,s}(z_j,\lambda_j)\lambda_j^{-2}\\
        &\leq 2K^{2+\frac{4}{p}}g_{q}(z_j,\lambda_j)\lambda_j^{-2}\\
        &\leq 4K^4g_{q}(z_j,\lambda_j)\lambda_j^{-2},
    \end{aligned}
    $$}
    where the last inequality follows from $\frac{1}{p}\leq \frac{1}{2}$. Therefore, we obtain
    $$
    d_i(z,w)\leq 2K^2d_j(z,w) \qquad \text{for any }z\in E(\Lambda),
    $$
    and hence, by \eqref{eq : estimate of d i and d j}, we have the conclusion.
    
    \textit{Case $(2$-$3)$, $(2$-$4)$, $(3$-$3)$, $(3$-$2)$ and $(3$-$4)$}. Proceed similarly to the proof above.

    \textit{Case $(4$-$4)$} Using Lemma \ref{lem : relation of lambda z and lambda tilde z in p,q,s-phase} and following the proof of Case 3 in \cite[Lemma 3.8]{Wontae2023a}, we have $r_i\leq 12 r_j$.

    \textit{Case $(4$-$2)$ and $(4$-$3)$}. Note that, by Lemma \ref{lem : relation of lambda z and lambda tilde z in p,q,s-phase},
    {$$
    \begin{aligned}
        g_{q}(z_i,\lambda_i)\lambda_i^{-2}&\leq g_{q,s}(z_i,\lambda_i)\lambda_i^{-2}\\
        &=g_{q,s}(z_j,\lambda_j)\lambda_i^{-2}\\
        &\leq 2^\frac{2}{p}g_{q,s}(z_j,\lambda_j)\lambda_j^{-2}
    \end{aligned}
    $$}
    and, similarly,
    { $$
        g_s(z_i,\lambda_i)\lambda_i^{-2}\leq 2^\frac{2}{p}g_{q,s}(z_i,\lambda_i)\lambda_j^{-2}.
    $$
    Thus, we obtain $d_i(z,w)\leq 2 d_j(z,w)$ and, therefore, conclude $r_i\leq 12 r_j$.}
\end{proof}
By this lemma, we get
\begin{equation}\label{Q(z_j) is contained in Q (z_i)}
    \frac{2}{K}Q_{r_j}(z_j)\subset 200 K Q_{r_i}(z_i)
\end{equation}
for all $j\in\mathcal{I}$. Using this, we summarize Lemma \ref{lem : relation of lambda z and lambda tilde z in p,s-phase}, Lemma \ref{lem : relation of lambda z and lambda tilde z in p,q-phase}, Lemma \ref{lem : relation of lambda z and lambda tilde z in p,q,s-phase} and Lemma \ref{lem : relation of lambda z and lambda tilde z in p-phase}.
\begin{lemma}   \label{lem : comparison of lambda_i and lambda_j}
    For any $i\in \mathbb{N}$ and $j\in\mathcal{I}$, we have $K^{-\frac{2}{p}}\lambda_j\leq \lambda_i\leq K^{\frac{2}{p}}\lambda_j$. Moreover, if $U_i=Q_i^{q,s}$, then $2^{-\frac{1}{p}}\lambda_j\leq \lambda_i\leq 2^\frac{1}{p}\lambda_j$.
\end{lemma}
 We show from the previous two lemmas that the measures of the neighboring cylinders are comparable.
\begin{lemma}   \label{lem : uniformly bounded of |U_i|/|U_j|} 
    There exists $c$ depending on $n$ and $K$ such that
    $$
    \sup_{\substack{i\in\mathbb{N} \\j\in \mathcal{I}}} \frac{|U_i|}{|U_j|}\leq c.
    $$
\end{lemma}
 \begin{proof}
    We divide it into $16$ cases as in Lemma \ref{lem : comparision of r_j and r_i}.
    
    \textit{Case $(1$-$1)$, $(1$-$2)$, $(1$-$3)$  $(1$-$4)$}. We obtain from $I_i\subset I_{r_i,\lambda_i}$, Lemma \ref{lem : comparision of r_j and r_i} and Lemma \ref{lem : comparison of lambda_i and lambda_j} that 
    $$
    \frac{|U_i|}{|U_j|}\leq \frac{|Q_i|}{|Q_j|}=\frac{2|B_1|r_i^{n+2}\lambda_i^{2-p}}{2|B_1|r_j^{n+2}\lambda_j^{2-p}}\leq \frac{(12K^2)^{n+2}r_j^{n+2}K^{\frac{2(p-2)}{p}}\lambda_j^{2-p}}{r_j^{n+2}\lambda_j^{2-p}}\leq (12K^2)^{n+3}
    $$
    for all $i\in\mathbb{N}$ and $j\in \mathcal{I}$.
    
    \textit{Case $(2$-$2)$, $(2$-$4)$, $(3$-$3)$, $(3$-$4)$ and $(4$-$4)$}. We conclude in a similar way from Lemma \ref{lem : comparision of r_j and r_i} and Lemma \ref{lem : comparison of lambda_i and lambda_j}.

    \textit{Case $(2$-$1)$}. We get from $a(z_i)\lambda_i^q\leq K^2\lambda_i^p$, $b(z_i)\lambda_i^s\leq K^2\lambda_i^p$, Lemma \ref{lem : comparision of r_j and r_i} and Lemma \ref{lem : comparison of lambda_i and lambda_j} that 
    \begin{align*}
        \frac{|U_i|}{|U_j|}&=\frac{|Q_i|}{|Q_j^{q}|}=\frac{2|B_1|r_i^{n+2}\lambda_i^2 g_q(z_j,\lambda_j)}{2|B_1|r_j^{n+2}\lambda_j^2\lambda_i^p}=\frac{r_i^{n+2}\lambda_i^2g_q(z_j,\lambda_j)}{r_j^{n+2}\lambda_j^2\lambda_i^p}\\
        &\leq \frac{(12K^2)^{n+2}K^{\frac{4}{p}}\lambda_j^2 r_j^{n+2}g_q(z_j,\lambda_j)}{\lambda_j^2r_j^{n+2}\lambda_i^p}=\frac{3(12K^2)^{n+3}g_q(z_j,\lambda_j)}{\lambda_i^p+\lambda_i^p+\lambda_i^p}\\
        &\leq \frac{2(12K^2)^{n+3} K^2\Lambda}{\lambda_i^p+a(z_i)\lambda_i^q+b(z_i)\lambda_i^s}\leq c(n,K)
    \end{align*}
    for all $i\in\mathbb{N}$ and $j\in \mathcal{I}$.
    
    \textit{Case $(3$-$1)$, $(4$-$1)$, $(3$-$2)$, $(2$-$3)$, $(4$-$2)$, $(4$-$3)$}. Similarly, we obtain the conclusion.
 \end{proof}
Now, we establish the inclusion property of $U_i$ and $U_j$ for any $i\in \mathbb{N}$ and $j\in \mathcal{I}$.
\begin{lemma}
    Let $i\in\mathbb{N}$ be such that $U_i=Q_i$. We have $\frac{2}{K}U_j\subset 50K^2 Q_i$ for every $j\in\mathcal{I}$.
\end{lemma}
\begin{proof}
    Since $\frac{2}{K} B_i \cap \frac{2}{K} B_j \neq \emptyset$ and $r_j\leq 12K^2 r_i$, we clearly obtain $\frac{2}{K}B_j\subset 50K B_i$. It remains to prove the inclusion in the time direction. As $I_j\subset I_{r_j,\lambda_j}(t_j)$, for $\tau \in \frac{2}{K}I_j$, we have from Lemma \ref{lem : comparision of r_j and r_i} and Lemma \ref{lem : comparison of lambda_i and lambda_j} that
    $$
    \begin{aligned}
        |\tau - t_i|&\leq |\tau - t_j|+|t_j - t_i|\leq 2\lambda_j^{2-p}\left(\frac{2}{K}r_j\right)^2+\lambda_i^{2-p}\left(\frac{2}{K}r_i\right)^2\\
        &\leq 2\left(K^{-\frac{2}{p}}\lambda_i\right)^{2-p}(24Kr_i)^2+\lambda_i^{2-p}(2r_i)^2\leq \lambda_i^{2-p}(50K^2r_i)^2 
    \end{aligned}
    $$
    and hence $\frac{2}{K}U_j \subset 50K^2 Q_i$.
\end{proof}
\begin{lemma}
    Let $i\in \mathbb{N}$ be such that, either $U_i=Q_i^q$ or $U_i=Q_i^s$ holds. Then we have $\frac{2}{K}U_j\subset 100K^3 U_i$ for every $j\in \mathcal{I}$.
\end{lemma}
\begin{proof}
    We may assume that $U_i=Q_i^q$. Since $\frac{2}{K} B_i \cap \frac{2}{K} B_j \neq \emptyset$ and $r_j\leq 12K^2 r_i$, clearly, we obtain $\frac{2}{K}B_j\subset 50K B_i$. It remains to prove the inclusion in the time direction. Since $I_j\subset I_{r_j,\lambda_j}(t_j)$, it is enough to check only when $U_j=Q_j$. For $\tau \in \frac{2}{K}I_{r_j,\lambda_j}(t_j)$, we have 
    { $$
    |\tau-t_i|\leq |\tau - t_j|+|t_j - t_i|\leq 2\lambda_j^{2-p}\left(\frac{2}{K}r_j\right)^2 + \frac{\lambda_i^2}{g_q(z_i,\lambda_i)}\left(\frac{2}{K}r_i\right)^2.
    $$}
    Since $K^2\lambda_j^p\geq a(z_j)\lambda_j^q$, $K^2\lambda_j^p\geq b(z_j)\lambda_j^s$ and Lemma \ref{lem : comparision of r_j and r_i} and Lemma \ref{lem : comparison of lambda_i and lambda_j} give $r_j\leq 12K^2r_i$ and $\lambda_j\leq K^{\frac{2}{p}}\lambda_i$, we obtain
    { $$
    \begin{aligned}
        \lambda_j^{2-p}\left(\frac{2}{K}r_j\right)^2&=3\frac{\lambda_j^2}{\lambda_j^p+\lambda_j^p+\lambda_j^p}\left(\frac{2}{K}r_j\right)^2\leq 12\frac{\lambda_j^2}{\lambda_j^p+a(z_j)\lambda_j^q+b(z_j)\lambda_j^s}r_j^2\\
        &\leq 12\frac{K^\frac{4}{p}\lambda_i^2}{\lambda_i^p+a(z_i)\lambda_i^q+b(z_i)\lambda_i^s}(12K^2r_i)^2\leq \frac{\lambda_i^2}{g_q(z_i,\lambda_i)}(48K^3r_i)^2.
    \end{aligned}
    $$
    Thus, we have 
    $$
    |\tau-t_i|\leq 2\frac{\lambda_i^2}{g_q(z_i,\lambda_i)}(48K^3r_i)^2 + \frac{\lambda_i^2}{g_q(z_i,\lambda_i)}\left(2r_i\right)^2\leq \frac{\lambda_i^2}{g_q(z_i,\lambda_i)}\left(100K^3 r_i\right)^2, 
    $$}
    and hence $\frac{2}{K}U_j \subset 100K^3 U_i$. 
\end{proof}
\begin{lemma}
    Let $i\in \mathbb{N}$ be such that $U_i=Q_i^{q,s}$. Then we have $\frac{2}{K}U_j\subset 200 U_i$ for each $j\in \mathcal{I}$.
\end{lemma}
\begin{proof}
    The proof can be obtained by using $r_j\leq 12 r_i$ and $2^{-\frac{1}{p}}\lambda_j\leq \lambda_i \leq 2^\frac{1}{p}\lambda_j$ instead of $(12K^2)^{-1} r_j\leq r_i\leq 12K^2 r_j$ and $K^{-\frac{2}{p}}\lambda_j\leq \lambda_i\leq K^\frac{2}{p}\lambda_j$ in the proof of the above lemma.
\end{proof}
Through the above three lemmas, we obtain the condition
\begin{equation}    \label{eq : U_j subset 200K^4 U_i}
    \frac{2}{K}U_j \subset K_iU_i\qquad \text{for all }i\in\mathbb{N} \text{ and } j\in \mathcal{I},
\end{equation}
where $K_i$ is defined in \eqref{EQQ2.5}. Finally, we prove that the cardinality of $\mathcal{I}$ is uniformly bounded.
\begin{lemma}\label{lem : bound of I}
    There exists a constant $c$ depending only on $n$ and $K$ such that $|\mathcal{I}|\leq c$ for every $i\in\mathbb{N}$.
\end{lemma}
\begin{proof}
    Since the cylinders $\left\{\frac{1}{6K^6}U_j\right\}_{j\in\mathbb{N}}$ are disjoint, we get \eqref{eq : U_j subset 200K^4 U_i} and Lemma \ref{lem : uniformly bounded of |U_i|/|U_j|} that 
    $$
    |200K^4 U_i|\geq \left|\bigcup_{j\in \mathcal{I}} \frac{1}{6K^6}U_j\right|=\sum_{j\in \mathcal{I}}\left|\frac{1}{6K^6}U_j\right|\geq \sum_{j\in \mathcal{I}} c(n,K) |U_i|=c(n,K)|\mathcal{I}||U_i|,
    $$
    and hence $|\mathcal{I}|\leq c(n,K)$.
\end{proof}

We summarize the above results below.
\begin{lemma}\label{LEM2.9}
    Let $K$ be as in \eqref{defn of K} and $E(\Lambda)$ as in \eqref{defn of E}. There exists a collection $\{\frac{1}{K}U_{i}\}_{i\in\mathbb{N}}$ of cylinders defined as in \eqref{EQQ2.4} satisfying the following properties:
	\begin{enumerate}[label=(\roman*),series=theoremconditions]
		\item\label{i} $\displaystyle E(\Lambda)^c\subset \bigcup_{i\in\mathbb{N}}\frac{1}{K}U_i$.
		\item\label{ii} $\displaystyle \tfrac{1}{6K^6}U_i\cap\tfrac{1}{6K^6}U_j= \emptyset$ for every $i,j\in \mathbb{N}$ with $i\ne j$.
		\item\label{iii} $3r_i\le d_i(U_i,E(\Lambda))\le 4r_i$ for every $i\in \mathbb{N}$.
		\item\label{iv} $4U_i\subset E(\Lambda)^c$ and $5U_i\cap E(\Lambda)\ne\emptyset$ for every $i\in \mathbb{N}$.
		\item\label{v} $(12K^2)^{-1}r_j\le r_i\le 12K^2 r_j$ for every $i\in \mathbb{N}$, $j\in \mathcal{I}$.
		\item\label{vi}$K^{-\frac{2}{p}}\lambda_j\le \lambda_i\le K^\frac{2}{p}\lambda_j$ for every $j\in \mathcal{I}$. 
  \item\label{Uij}There exists a constant $c = c(n,K)$ such that $|U_i|\leq c|U_j|$, for every $i\in \mathbb{N}$ and $j\in \mathcal{I}$.
 \item\label{vii} $\frac{2}{K}U_j \subset K_iU_i$ for every $i\in \mathbb{N}$, $j\in \mathcal{I}$.
 
      \item\label{UUi} If $U_i=Q_i$, then there exists a constant $c=c([a]_\alpha,\alpha,K)$ such that $r_i^\alpha\lambda_i^q\le c\lambda_i^p$ and $r^{\beta}_i \lambda^s_i\leq c \lambda^p_i.$

		\item\label{viii} If $U_i=Q^q_{i}$, then $\tfrac{a(z_i)}{2}\le a(z)\le 2a(z_i)$ for every $z\in 200K Q_{r_i}(z_i)$. If $U_i=Q^s_i,$ then $\tfrac{b(z_i)}{2}\le b(z)\le 2b(z_i)$ for every $z\in 200K Q_{r_i}(z_i)$ and if $U_i=Q^{q,s}_i,$ then $\tfrac{a(z_i)}{2}\le a(z)\le 2a(z_i)$ and $\tfrac{b(z_i)}{2}\le b(z)\le 2b(z_i)$ for every $z\in 200K Q_{r_i}(z_i).$ 
		\item\label{ix} For any $i\in\mathbb{N}$, the cardinality of $\mathcal{I}$, denoted by $|\mathcal{I}|$, is finite. Moreover, there exists a constant $c= c(n,K)$, such that $|\mathcal{I}|\le c$.
	\end{enumerate}
\end{lemma}
\subsection{Partition of unity} The following lemma demonstrates the construction of a partition of unity subordinate to Whitney decomposition $\{\frac{2}{K}U_i\}_{i\in \mathbb{N}}.$
\begin{lemma}\label{lem : partition of unity}
There exists a partition of unity $\{\omega_i\}_{i \in \mathbb{N}}$ subordinate to the Whitney decomposition $\{\frac{2}{K}U_i\}_{i \in \mathbb{N}}$ with the following properties:
\begin{enumerate}[label=(\roman*),series=theoremconditions]
\item $0\leq \omega_i \leq 1,$ $\omega_i \in C^{\infty}_0(\frac{2}{K}U_i)$ for every $i \in \mathbb{N}$ and $\sum_{i \in \mathbb{N}}\omega_i=1$ on $E(\Lambda)^c.$
\item There exists a constant $c=c(n, K)$ such that $||\nabla \omega_j||_{\infty}\leq cr^{-1}_i,$ for every $i \in \mathbb{N}$ and $j \in \mathcal{I}.$
\item There exists a constant $c=c(n, K)$ such that
{ \begin{align*}
    ||\partial_t\omega_j||_{\infty}\leq \begin{cases}
        c r^{-2}_i \lambda^{p-2}_i,\,\,\, &\text{if}\,\,\, U_i=Q_i,\\
        cr^{-2}_ig_q(z_i,\lambda_i)\lambda^{-2}_i,\,\,\, &\text{if}\,\,\, U_i=Q^q_i,\\
        cr^{-2}_ig_s(z_i,\lambda_i)\lambda^{-2}_i,\,\,\, &\text{if}\,\,\, U_i=Q^s_i,\\
        cr^{-2}_i\Lambda \lambda^{-2}_i,\,\,\, &\text{if}\,\,\, U_i=Q^{q,s}_i
    \end{cases}
\end{align*}}
for any $i\in\mathbb{N}$ and $j\in\mathcal{I}.$
\end{enumerate}
\end{lemma}
\begin{proof}
    In the previous subsection, we obtain the Whitney decomposition $\{\frac{2}{K}U_i\}_{i\in\mathbb{N}}$. Then, for each $i\in\mathbb{N}$, we choose $\psi_i\in C^\infty_0(\frac{2}{K} U_i)$ satisfying
    $$
    0\leq \psi_i\leq 1,\quad \psi_i\equiv 1\text{ in }U_i,\quad \|\nabla \psi_i\|_{\infty}\leq \frac{2}{K}r_i^{-1}
    $$
    and
    {\begin{align*}
    ||\partial_t\psi_i||_{\infty}\leq \begin{cases}
        \frac{2}{K} r^{-2}_i \lambda^{p-2}_i, \,\,\, &\text{if}\,\,\, U_i=Q_i, \\
        \frac{2}{K} r^{-2}_ig_q(z_i,\lambda_i)\lambda^{-2}_i, \,\,\, &\text{if}\,\,\, U_i=Q^q_i, \\
        \frac{2}{K} r^{-2}_ig_s(z_i,\lambda_i)\lambda^{-2}_i, \,\,\, &\text{if}\,\,\, U_i=Q^s_i, \\
        \frac{2}{K} r^{-2}_i\Lambda \lambda^{-2}_i, \,\,\, &\text{if}\,\,\, U_i=Q^{q,s}_i.
    \end{cases}
\end{align*}}
Since $E(\Lambda)^c\subset \bigcup_{i\in \mathbb{N}}\frac{2}{K}U_i$ and $|\mathcal{I}|$ is finite for each $i\in\mathbb{N}$, the function
$$
\omega_i(z)=\frac{\psi_i(z)}{\sum_{j\in\mathbb{N}}\psi_j(z)}=\frac{\psi_i(z)}{\sum_{j\in\mathcal{I}}\psi_j(z)}
$$
is well-defined and satisfies
$$
\omega_i\in C^\infty_0\left(\frac{2}{K}U_i\right),\quad 0\leq \omega_i\leq 1 \text{ in }\frac{2}{K}U_i,\quad \sum_{i\in\mathbb{N}}\omega_i(z)=1.
$$
Thus, the collection $\{\omega_i\}_{i\in\mathbb{N}}$ is a partition of unity subordinate to $\{\frac{2}{K}U_i\}_{i\in \mathbb{N}}$. By Lemmas \ref{lem : comparision of r_j and r_i}, \ref{lem : comparison of lambda_i and lambda_j} and \ref{lem : bound of I}, we have
$$
||\nabla \omega_j||_{\infty}\leq cr^{-1}_i\quad \text{and}\quad \begin{aligned}
    ||\partial_t\omega_j||_{\infty}\leq \begin{cases}
        c r^{-2}_i \lambda^{p-2}_i\,\,\, &\text{if}\,\,\, U_i=Q_i\\
        cr^{-2}_ig_q(z_i,\lambda_i)\lambda^{-2}_i\,\,\, &\text{if}\,\,\, U_i=Q^q_i\\
        cr^{-2}_ig_s(z_i,\lambda_i)\lambda^{-2}_i\,\,\, &\text{if}\,\,\, U_i=Q^s_i\\
        cr^{-2}_i\Lambda \lambda^{-2}_i\,\,\, &\text{if}\,\,\, U_i=Q^{q,s}_i
    \end{cases}
\end{aligned}
$$
for any $i\in\mathbb{N}$ and $j\in\mathcal{I}$.
\end{proof}

\section{\bf Construction of test function via Lipschitz truncation} \label{sec : Construction of test function via Lipschitz truncation}
The main goal of this section is to construct a Lipschitz function $v^{\Lambda}_h$ which can be used as a test function in the proof of energy estimate  Theorem \ref{thm : the Caccioppoli inequality}. We begin by defining $v_h$ and $v^{\Lambda}_h,$ establishing a crucial Poincaré-type inequality for $v_h.$ By combining this result with the properties of the Whitney decomposition, we conclude the Lipschitz regularity of $v^{\Lambda}_h.$
\subsection{Definition of test function}    \label{subsection : Definition of test function}
Take $f=\chi_{U_{R_2,S_2}(z_0)}\left(|\nabla u|+|u-u_0|+|F| \right)\in L^{p}(\mr^{n+1})$ and $u_0=(u)_{U_{R_2,S_2}(z_0)}$, where $\chi$ is the characteristic function and $u, F$ are extended to zero outside $U_{R_2,S_2}(z_0)$.

Let $0<h_0<\frac{S_2-S_1}{4}$ be a sufficiently small number, and let $\eta \in C^\infty_0 (B_{R_2}(x_0))$ and $\zeta\in C^\infty_0 (\ell_{S_2-h_0}(t_0))$ be standard cutoff functions satisfying $0\leq \eta \leq 1$, $0\leq \zeta \leq 1$, $\eta\equiv 1$ in $B_{R_1}(x_0)$, $\zeta\equiv 1$ in $\ell_{S_1}(t_0)$ and 
\begin{equation}    \label{eq : bound of gradients of eta and zeta}
    \|\nabla \eta\|_\infty \leq \frac{3}{R_2-R_1}, \quad \|\partial_t \zeta\|_\infty\leq \frac{3}{S_2-S_1}.
\end{equation}
Now for $0<h<h_0,$ we define the truncated solution
\begin{align}\label{test fn_h}
    v_h(z)=[u(z)- u_0]_h \eta(x)\zeta(t), \quad z=(x,t) \in \mathbb{R}^{n+1},
\end{align}
as a suitable candidate for test function to be used to prove energy estimates. For $z\in \mathbb{R}^{n+1},$ we define the Lipschitz truncation of $v_h$ as
\begin{align}\label{lip trunc_h}
    v^{\Lambda}_h(z)=v_h(z)-\sum_{i \in \mathbb{N}}\left(v_h(z)-v^i_h\right)\omega_i(z),
\end{align}
where
\begin{align*}
    v^i_h=\begin{cases}
        \displaystyle \miint{\frac{2}{K}U_i} v_h(z)dz, \,\,\,\,\, &\text{if}\,\,\, \frac{2}{K}U_i\subset U_{R_2, S_2}(z_0),\\
        \displaystyle 0, &\text{elsewhere}.
    \end{cases}
\end{align*}
Similarly, we define
\begin{align}\label{test fn}
    v(z)=(u(z)- u_0)\eta(x)\zeta(t)\,\,\,\, \text{and}\,\,\,\, v^{\Lambda}(z)=v(z)-\sum_{i\in \mathbb{N}}(v(z)-v^i)\omega_i(z),
\end{align}
where
\begin{align*}
    v^i=\begin{cases}
        \displaystyle \miint{\frac{2}{K}U_i}v(z)dz, \,\,\,\,\, &\text{if}\,\,\, \frac{2}{K}U_i\subset U_{R_2, S_2}(z_0),\\
        \displaystyle 0, &\text{elsewhere}.
    \end{cases}
\end{align*}
\subsection{Preliminary lemmas}
In this subsection, we discuss some preparatory machinery to prove Lemma \ref{LEMMA: PTI}.  We denote a family of parameters as
$$
\operatorname{data}\equiv \operatorname{data}(n,N,p,q,s,\alpha,\beta,\nu,L,\|a\|_{L^\infty},\|b\|_{L^\infty},[a]_\alpha,[b]_\beta,R_1,R_2,S_1,S_2,K).
$$
First, we recall the following lemma from \cite[Lemma 8.1]{V_Boglein_phd_thesis}.
\begin{lemma}\label{lem: steklov average}
    Let $f \in L^1(\Omega_T)$ and $h>0.$ Then there exists a constant $c=c(n)$ such that 
    \begin{align*}
        \miint{U_{r_1, r_2}(z_0)}[f]_h\, dz \leq c\miint{[U_{r_1, r_2}(z_0)]_h} f\, dz,
    \end{align*}
    where $[U_{r_1, r_2}(z_0)]_h=U_{r_1, r_2+h}(z_0).$
\end{lemma}
\begin{lemma}\label{lem : miint f^gamma bound}
    Let $1\leq \gamma \leq d.$ Then for any cylinder $Q\subset \mathbb{R}^{n+1}$ such that $Q\cap E(\Lambda)\neq \emptyset,$ we have
    \begin{align}\label{EQQ3.1}
        \miint{Q}f^{\gamma} \,dz\leq \Lambda^{\frac{\gamma}{p}}.
    \end{align}
    Moreover, there exists a constant $c=c(\operatorname{data})$ such that
    \begin{align}\label{EQQ3.2}
        \miint{4K_iU_i}f^{\gamma}\, dz\leq c\lambda^{\gamma}_{i},
    \end{align}
    where $K_iU_i$ is defined in \eqref{EQQ2.4}-\eqref{EQQ2.5}.
\end{lemma}
\begin{proof}
Let $w \in Q\cap E(\Lambda).$ Then by \eqref{max fn} and \eqref{defn of E}, we have
\begin{align*}
    \miint{Q}f^{\gamma} \, dz \leq\left(\miint{Q}f^d\, dz\right)^{\frac{\gamma}{d}}\leq \left(M (f^d)(w)\right)^{\frac{\gamma}{d}}&\leq \left(M(f^d+(a f^q)^{\frac{d}{p}}+(b f^s)^{\frac{d}{p}})(w)\right)^{\frac{\gamma}{d}}\\
    &\leq \Lambda^{\frac{\gamma}{p}}
\end{align*}
and that proves \eqref{EQQ3.1}. By Lemma \ref{LEM2.9} \ref{iv}, we have $4K_iU_i \cap E(\Lambda)\neq \emptyset$ and it follows that
\begin{align*}
    \miint{4K_i U_i}f^{\gamma}dz \leq \Lambda^{\frac{\gamma }{p}}.
\end{align*}
\noindent {\bf Case I:} $U_i=Q_i.$ In this case, we have $K^2 \lambda^p_i\geq a(z_i)\lambda^q_i$ and $K^2 \lambda^p_i\geq b(z_i)\lambda^s_i.$ Using this, we have 
\begin{align*}
    \Lambda=\lambda^p_i+a(z_i)\lambda^q_i+b(z_i)\lambda^s_i \leq (2K^2+1)\lambda^p_i.
\end{align*}
Hence, we get
\begin{align*}
 \miint{4K_i U_i}f^{\gamma}dz \leq \Lambda^{\frac{\gamma }{p}}\leq (2K^2+1)^{\frac{\gamma}{p}}\lambda^{\gamma}_i=c\lambda^{\gamma}_i   
\end{align*}
and this proves \eqref{EQQ3.2} for this case.

\noindent {\bf Case II:} $U_i=Q^q_i.$ In this case, we have $K^2 \lambda^p_i< a(z_i)\lambda^q_i$ and $K^2 \lambda^p_i\geq b(z_i)\lambda^s_i.$ By Lemma \ref{LEM2.9} \ref{viii}, we also have $\frac{a(z_i)}{2}\leq a(z)\leq 2a(z_i)$ for every $z \in 200K Q_{r_i}(z_i).$ Then we have
\begin{align}\label{EQQ3.3}
    (a(z_i))^{\frac{d}{p}}\miint{4K_i U_i}(f^q)^{\frac{d}{p}}\, dz= \miint{4K_iU_i}(a(z_i)f^q)^{\frac{d}{p}}\, dz \leq 2^{\frac{d}{p}}\miint{4K_iU_i}(a(z)f^q)^{\frac{d}{p}}\, dz.
\end{align}
By Lemma \ref{LEM2.9} \ref{iv}, there exists a $w \in 4K_i U_i \cap E(\Lambda)$ and we obtain
\begin{align*}
    \miint{4K_i U_i}(a(z)f^q)^{\frac{d}{p}}\, dz \leq M\left((a f^q)^{\frac{d}{p}}\right)(w)\leq \left(M(f^d+(a f^q)^{\frac{d}{p}}+(b f^s)^{\frac{d}{p}})(w)\right) \leq \Lambda^{\frac{d}{p}}.
\end{align*}
Hence, from \eqref{EQQ3.3}, we have
\begin{align*}
 a(z_i)^{\frac{d}{p}}\miint{4K_i U_i}(f^q)^{\frac{d}{p}}\, dz \leq c\Lambda^{\frac{d}{p}}&=c\left(\lambda^p_i+a(z_i)\lambda^q_i+b(z_i)\lambda^s_i\right)^{\frac{d}{p}}\\
 &\leq c\left(2a(z_i)\lambda^q_i+K^2 \lambda^p_i\right)^{\frac{d}{p}}\leq (3c)^{\frac{d}{p}} (a(z_i)\lambda^q_i)^{\frac{d}{p}}.
\end{align*}
Since $a(z_i)>0, $ we get
\begin{align*}
  \miint{4K_i U_i}(f^q)^{\frac{d}{p}}\, dz \leq c \lambda^{\frac{q d}{p}}_i  
\end{align*}
and finally, we have 
\begin{align*}
    \miint{4K_i U_i}f^{\gamma}\, dz \leq \left(\miint{4K_iU_i}f^{\frac{q d}{p}}{ dz}\right)^{\frac{\gamma p}{q d}}\leq c \lambda^{\gamma}_i,
\end{align*}
which proves \eqref{EQQ3.2} for this case.

\noindent {\bf Case III:} $U_i=Q^s_i.$ In this case, we have $K^2 \lambda^p_i\geq a(z_i)\lambda^q_i$ and $K^2 \lambda^p_i< b(z_i)\lambda^s_i.$ By Lemma \ref{LEM2.9} \ref{viii}, we also have $\frac{b(z_i)}{2}\leq b(z)\leq 2b(z_i)$ for every $z \in 200K Q_{r_i}(z_i).$ Then we have
\begin{align}\label{EQQ3.3}
    (b(z_i))^{\frac{d}{p}}\miint{4K_i U_i}(f^s)^{\frac{d}{p}}\, dz= \miint{4K_iU_i}(b(z_i)f^s)^{\frac{d}{p}}\, dz \leq 2^{\frac{d}{p}}\miint{4K_iU_i}(b(z)f^s)^{\frac{d}{p}}\, dz.
\end{align}
By Lemma \ref{LEM2.9} \ref{iv}, there exists a $w \in 4K_i U_i \cap E(\Lambda)$ and we obtain
\begin{align*}
    \miint{4K_i U_i}(b(z)f^s)^{\frac{d}{p}}\, dz \leq M\left((b f^s)^{\frac{d}{p}}\right)(w)\leq \left(M(f^d+(a f^q)^{\frac{d}{p}}+(b f^s)^{\frac{d}{p}})(w)\right) \leq \Lambda^{\frac{d}{p}}.
\end{align*}
Hence, from \eqref{EQQ3.3}, we have
\begin{align*}
 b(z_i)^{\frac{d}{p}}\miint{4K_i U_i}(f^q)^{\frac{d}{p}}\, dz \leq c\Lambda^{\frac{d}{p}}&=c\left(\lambda^p_i+a(z_i)\lambda^q_i+b(z_i)\lambda^s_i\right)^{\frac{d}{p}}\\
 &\leq c\left(2b(z_i)\lambda^q_i+K^2 \lambda^p_i\right)^{\frac{d}{p}}\leq (3c)^{\frac{d}{p}} (b(z_i)\lambda^q_i)^{\frac{d}{p}}.
\end{align*}
Since $b(z_i)>0, $ we get
\begin{align*}
  \miint{4K_i U_i}(f^s)^{\frac{d}{p}}\, dz \leq c \lambda^{\frac{s d}{p}}_i  
\end{align*}
and finally, we have 
\begin{align*}
    \miint{4K_i U_i}f^{\gamma}\, dz \leq \left(\miint{4K_iU_i}f^{\frac{s d}{p}}{ dz}\right)^{\frac{\gamma p}{s d}}\leq c \lambda^{\gamma}_i,
\end{align*}
which proves \eqref{EQQ3.2} for this case.

\noindent {\bf Case IV:} $U_i=Q^{q,s}_i.$ In this case, we have $K^2 \lambda^p_i< a(z_i)\lambda^q_i$ and $K^2 \lambda^p_i< b(z_i)\lambda^s_i.$ By Lemma \ref{LEM2.9} \ref{viii}, we also have $\frac{a(z_i)}{2}\leq a(z)\leq 2a(z_i)$ and  $\frac{b(z_i)}{2}\leq b(z)\leq 2b(z_i)$ for every $z \in 200K Q_{r_i}(z_i).$ Now we consider two cases: either $a(z_i)\lambda^q_i \leq b(z_i)\lambda^s_i$ or $a(z_i)\lambda^q_i \geq b(z_i)\lambda^s_i.$ The first case corresponds to case III and the second case corresponds to case II and we arrive at the same conclusion.
\end{proof}
Now we prove a parabolic Poincar\'{e} type result.
\begin{lemma}\label{LEM3.2}
Let $U=B_{r_1}\times \ell_{r_2}$ be any cylinder defined in \eqref{eq : definition of U and Q} and \eqref{eq : definition of ell and I_r} satisfying $B_{r_1}\subset B_{R_2}(x_0).$ The the following estimates hold:
\begin{enumerate}[label=(\roman*),series=theoremconditions]
\item There exists a constant $c=c(\operatorname{data})$ such that
\begin{align}\label{EQQ3.5}
\miint{U}|v_h-(v_h)_{U}|\, dz &\leq c \frac{r_2}{r_1}\miint{[U]_h}\left(f^{p-1}+a(z)f^{q-1}+b(z)f^{s-1}\right)\, dz\nonumber\\
&\quad+ c(r_1+r_2)\miint{[U]_h}f \, dz.    
\end{align}
\item If in addition $\ell_{r_2}\cap \ell_{S_2}(t_0)^c\neq \emptyset,$ then there exists a constant $c=c(\operatorname{data})$ such that 
\begin{align}\label{EQQ3.6}
    \miint{U}|v_h|\, dz &\leq c \frac{r_2}{r_1}\miint{[U]_h}\left(f^{p-1}+a(z)f^{q-1}+b(z)f^{s-1}\right)\, dz\nonumber\\
&\quad+ c(r_1+r_2)\miint{[U]_h}f \, dz.
\end{align}
\end{enumerate}
In addition, the above estimates hold with $v_h$ and $[U]_h$ replaced by $v$ and $U.$
\end{lemma}
\begin{proof}
For $t_1,t_2\in \ell_{r_2},$ $t_1\leq t_2$, let $\zeta_\delta \in W^{1,\infty}_0(\ell_{r_2})$ be a piecewise linear cut-off function defined by
$$
\zeta_\delta (t)=\left\{
\begin{aligned}
    &0,\qquad\qquad\quad t\in (-\infty,t_1-\delta),\\
    &1+\frac{t-t_1}{\delta},\,\quad t\in[t_1-\delta,t_1],\\
    &1,\qquad\qquad\quad t\in (t_1,t_2),\\
    &1-\frac{t-t_2}{\delta},\,\quad t\in[t_2,t_2+\delta],\\
    &0,\qquad\qquad\quad t\in (t_2+\delta,\infty).
\end{aligned}
\right.
$$
Furthermore, let $\varphi\in C_0^\infty(B_{r_1})$ be a nonnegative function satisfying
$$
\dashint_{B_{r_1}} \varphi\, dx=1,\quad \|\nabla \varphi\|_\infty \leq \frac{c(n)}{r_1},\quad \|\varphi\|_\infty \leq c(n).
$$
By taking a test function $\psi=\varphi \eta \zeta \zeta_\delta \in W^{1,\infty}_0(U\cap U_{R_2,S_2-h}(z_0))$ in the Steklov averaged weak formulation of \eqref{eq: main equation}, where standard cutoff functions $\eta$ and $\zeta$ are defined in Subsection \ref{subsection : Definition of test function}, we obtain 
$$
\iint_{U} -[u-u_0]_h \cdot \varphi \eta \partial_t(\zeta \zeta_\delta)\, dz=\iint_U [-\mathcal{A}(\cdot,\nabla u)+\mathcal{B}(\cdot, F)]_h \cdot \nabla \psi\, dz.
$$
By the definition of $f$, \eqref{eq: growth condition of A} and \eqref{eq: growth condition of B}, we have
$$
\begin{aligned}
    \mathrm{I}&=\left|\iint_{U} -[u-u_0]_h \cdot \varphi \eta \zeta\partial_t\zeta_\delta\, dz\right|\\
    &\leq \left|\iint_{U} [f]_h \cdot\eta\varphi \zeta_\delta \partial_t \zeta \, dz\right|+c(L)\left|\iint_U [f^{p-1}+a(\cdot)f^{q-1}+b(\cdot) f^{s-1}]_h \cdot\nabla \psi\, dz\right|\\
    &=\mathrm{I}\mathrm{I}+\mathrm{I}\mathrm{I}\mathrm{I}.
\end{aligned}
$$
First, we get from Lemma \ref{lem: steklov average} that 
$$
\mathrm{I}\mathrm{I}\leq \iint_{U} [f]_h |\partial_t \zeta|\|\varphi\|_\infty \, dz \leq c(n,S_1,S_2)r_1^n r_2 \miint{[U]_h} f\, dz
$$
and
$$
\begin{aligned}
    \mathrm{I}\mathrm{I}\mathrm{I}&\leq c(L)\iint_U [f^{p-1}+a(\cdot)f^{q-1}+b(\cdot) f^{s-1}]_h |\nabla \psi|\, dz\\
    &\leq c(L)\iint_U [f^{p-1}+a(\cdot)f^{q-1}+b(\cdot) f^{s-1}]_h |\varphi\nabla \eta +\eta \nabla \varphi|\, dz\\
    &\leq c(L)\iint_U [f^{p-1}+a(\cdot)f^{q-1}+b(\cdot) f^{s-1}]_h (\varphi +\eta)|\nabla \eta +\nabla \varphi|\, dz\\
    &\leq c(n,L)r_1^n r_2\left(\frac{1}{R_2-R_1}+\frac{1}{r_1}\right)\miint{[U]_h} (f^{p-1}+a(z)f^{q-1}+b(z)f^{s-1})\, dz.
\end{aligned}
$$
Note that $r_1\leq R_2$, and hence
$$
\frac{1}{R_2-R_1}+\frac{1}{r_1}\leq 2\max\left\{\frac{R_2}{R_2-R_1},1\right\}\frac{1}{r_1}.
$$
Therefore, we have
$$
\mathrm{I}\mathrm{I}\mathrm{I}\leq c(n,L,R_1,R_2)r_1^{n-1} r_2\miint{[U]_h} (f^{p-1}+a(z)f^{q-1}+b(z)f^{s-1})\, dz.
$$
Next, we obtain from the one-dimensional Lebesgue differentiation theorem that 
$$
\begin{aligned}
    \lim_{\delta\rightarrow 0^+} \mathrm{I}&=\left|\lim_{\delta\rightarrow 0^+} \iint_{U} [u-u_0]_h \cdot \varphi \eta \zeta\partial_t\zeta_\delta \, dz\right|\\
    &=\left|\int_{B_{r_1} \times \{t_1\}} v_h \varphi\, dz-\int_{B_{r_1} \times \{t_2\}} v_h \varphi\, dz\right|\\
    &= |B_1|{r_1}^n |(v_h \varphi)_{B_{r_1}}(t_1)-(v_h\varphi)_{B_{r_1}}(t_2)|.
\end{aligned}
$$
Combining the above inequalities, we conclude that 
\begin{equation}    \label{eq : essential supremum of v_h varphi}
\begin{aligned}
    &\operatorname*{ess\,sup}_{t_1,t_2\in \ell_{r_2}} |(v_h \varphi)_{B_{r_1}}(t_1)-(v_h\varphi)_{B_{r_1}}(t_2)|\leq c r_2 \miint{[U]_h} f\, dz\\
    &\qquad\qquad +cr_1^{-1}r_2 \miint{[U]_h} (f^{p-1}+a(z)f^{q-1}+b(z)f^{s-1})\, dz,
\end{aligned}
\end{equation}
where $c$ depends only on $n,L,S_1,S_2,R_1$ and $R_2$.

Now, we estimate the left-hand sides of \eqref{EQQ3.5} and \eqref{EQQ3.6}. First, we prove \eqref{EQQ3.5}. By the standard Poincar\'{e} inequality, we have
\begin{equation}\label{eq : estimate v_h-(v_h)_U}
\begin{aligned}
    \miint{U} |v_h-(v_h)_U|\, dz &\leq \miint{U} |v_h - (v_h)_{B_{r_1}}|\, dz + \miint{U} |(v_h)_{B_{r_1}} - (v_h)_U|\, dz\\
    &\leq c(n){r_1} \miint{U} |\nabla v_h|\, dz +\miint{U}|(v_h)_{B_{r_1}} - (v_h)_U|\, dz.
\end{aligned}
\end{equation}
It follows from the definition of $f$, Lemma \ref{lem: steklov average} and \eqref{eq : bound of gradients of eta and zeta} that
\begin{equation}    \label{eq : estimate nabla v_h}
    \miint{U} |\nabla v_h|\, dz=\miint{U} |[\nabla u]_h \eta \zeta+[u-u_0]_h \nabla \eta \zeta|\, dz\leq c(n,R_1,R_2)\miint{[U]_h} f \, dz.
\end{equation}
Moreover, we obtain
\begin{equation}    \label{eq : estimate 1 (v_h)_B(tau)-(v_h)_U}
\begin{aligned}
    \miint{U}|(v_h)_{B_{r_1}}(\tau) - (v_h)_U|\, dz&=\dashint_{\ell_{r_2}}|(v_h)_{B_{r_1}}(\tau) - (v_h)_U|\, d\tau\\
    &=\dashint_{\ell_{r_2}}\left|\dashint_{\ell_{r_2}} (v_h)_{B_{r_1}}(\tau) - (v_h)_{B_{r_2}}(\sigma)\, d\sigma\right|\, d\tau\\
    &\leq 2\dashint_{\ell_{r_2}} |(v_h)_{B_{r_1}}(\tau)-(v_h \varphi)_{B_{r_1}}(\tau)|\, d\tau\\
    &\qquad + \operatorname*{ess\,sup}_{t_1,t_2\in \ell_{r_2}} |(v_h \varphi)_{B_{r_1}}(t_1)-(v_h\varphi)_{B_{r_1}}(t_2)|.
\end{aligned}
\end{equation}
To estimate the first term, we use the fact that $(\varphi)_{B_{r_1}}=1$, the standard Poincar\'{e} inequality and \eqref{eq : estimate nabla v_h} to obtain
\begin{equation}    \label{eq : estimate 2 (v_h)_B-(v_h)_U}
\begin{aligned}
    \dashint_{\ell_{r_2}} |(v_h)_{B_{r_1}}(\tau)-(v_h \varphi)_{B_{r_1}}(\tau)|\, d\tau&=\dashint_{\ell_{r_2}}\left|(v_h)_{B_{r_1}}(\tau)\dashint_{B_{r_1}} \varphi\, dx-\dashint_{B_{r_1}\times \{\tau\}}v_h\varphi \, dx\right|\, d\tau\\
    &=\dashint_{\ell_{r_2}}\left|\dashint_{B_{r_1}\times \{\tau\}} \varphi(v_h-(v_h)_{B_{r_1}}) \, dx\right|\, d\tau\\
    &\leq c(n)\|\varphi\|_\infty r_1\miint{U} |\nabla v_h|\, dz\\
    &\leq c(n,R_1,R_2) r_1\miint{[U]_h} f\, dz.
\end{aligned}
\end{equation}
Combining \eqref{eq : essential supremum of v_h varphi}, \eqref{eq : estimate v_h-(v_h)_U}, \eqref{eq : estimate nabla v_h}, \eqref{eq : estimate 1 (v_h)_B(tau)-(v_h)_U} and \eqref{eq : estimate 2 (v_h)_B-(v_h)_U}, we have the conclusion.

Finally, we prove \eqref{EQQ3.6}. As in \eqref{eq : estimate v_h-(v_h)_U},  we obtain
\begin{equation}    \label{eq : estimate v_h}
    \begin{aligned}
        \miint{U} |v_h|\,dz &\leq \miint{U} |v_h-(v_h)_{B_{r_1}}|\, dz +\dashint_{\ell_{r_2}} |(v_h)_{B_{r_1}}|\, d\tau\\
        &\leq cr_1 \miint{U} |\nabla v_h|\, dz + \dashint_{\ell_{r_2}} |(v_h)_{B_{r_1}}|\, d\tau.
    \end{aligned}
\end{equation}
Then we have
\begin{equation}    \label{eq : estimate (v_h)_B_r}
    \dashint_{\ell_{r_2}} |(v_h)_{B_{r_1}}|\, d\tau\leq \dashint_{\ell_{r_2}} |(v_h)_{B_{r_1}}-(v_h\varphi)_{B_{r_1}}|\, dt + \dashint_{\ell_{r_2}} |(v_h\varphi)_{B_{r_1}}|\, dt.
\end{equation}
Since $\ell_{r_2}\cap \ell_{S_2}(t_0)^c\neq \emptyset$, there exists $t_2\in \ell_{r_2}$ such that $t_2\in \ell_{S_2}(t_0)^c$. Since $\zeta\equiv 0$ in $\ell_{S_2}(t_0)$, we have
\begin{equation}\label{eq : estimate (v_h varphi)_B_r_1}
    \dashint_{\ell_{r_2}} |(v_h\varphi)_{B_{r_1}}|\, dt\leq \operatorname*{ess\,sup}_{t\in \ell_{r_2}} |(v_h \varphi)_{B_{r_1}}(t)|\leq \operatorname*{ess\,sup}_{t_1,t_2\in \ell_{r_2}} |(v_h \varphi)_{B_{r_1}}(t_1)-(v_h\varphi)_{B_{r_1}}(t_2)|.
\end{equation}
Thus, combining \eqref{eq : essential supremum of v_h varphi}, \eqref{eq : estimate nabla v_h}, \eqref{eq : estimate 2 (v_h)_B-(v_h)_U}, \eqref{eq : estimate v_h}, \eqref{eq : estimate (v_h)_B_r} and \eqref{eq : estimate (v_h varphi)_B_r_1}, we have the conclusion.
\end{proof}
Next, we recall the boundary version of the Poincar\'{e} inequality from \cite[Theorem 6.22]{Kinnunen2021}.
\begin{lemma}\label{lem: boundary poincare}
Let $B_{\rho}(x_0)\subset\mathbb{R}^n$ and $B_{r}\subset\mathbb{R}^n$ be balls that satisfy $B_r\cap B_{\rho}(x_0)^c\ne\emptyset$. Assume that $v\in W_0^{1,\eta}(B_{\rho}(x_0))$ with $1<\eta<\infty$.  Moreover, let $1\le \sigma \le \tfrac{n\eta}{n-\eta}$ for $1<\eta<n$ and $1\le \sigma<\infty$ for $n\le \eta<\infty$. Then there exists a constant $c=c(n,\eta,\sigma)$ such that
	\[
		\left(\dashint_{B_{4r}}|v|^\sigma\,dx\right)^\frac{1}{\sigma}\le cr \left(\dashint_{B_{4r}}|\nabla v|^\eta\,dx\right)^\frac{1}{\eta}.
	\]    
\end{lemma}

\subsection{Poincar\'{e} type inequality for the test function}
In this subsection, we prove the Poincar\'{e} type inequality for $v_h.$
\begin{lemma} \label{LEMMA: PTI}
Let $K_i U_i$ be defined in \eqref{EQQ2.5}. Then the following estimates hold:
\begin{enumerate}[label=(\roman*),series=theoremconditions]
\item If $K_iU_i \subset U_{R_2, S_2}(z_0)$, then there exists a constant $c$ such that
\begin{align}\label{EQQ3.7}
    &\miint{K_iU_i}|v_h-(v_h)_{K_iU_i}|\, dz \leq c(\operatorname{data}, \Lambda)r_i 
\end{align}
and
\begin{align}\label{EQQ3.8}
    &\miint{K_iU_i}|v-(v)_{K_iU_i}|\, dz \leq c(\operatorname{data})\lambda_ir_i.
\end{align}
\item If $K_iU_i \not\subset U_{R_2, S_2}(z_0),$ then there exists a constant $c$ such that
\begin{align}\label{EQQ3.9}
    &\miint{K_iU_i}|v_h|\, dz \leq c(\operatorname{data}, \Lambda)r_i
\end{align}
and
\begin{align}\label{EQQ3.10}
    &\miint{K_iU_i}|v|\, dz \leq c(\operatorname{data})\lambda_ir_i. 
\end{align}
\end{enumerate}
\end{lemma}
\begin{proof}
 We start by proving \eqref{EQQ3.8}. Note that we assume $K_iU_i \subset U_{R_2, S_2}(z_0)$, and hence we can apply Lemma \ref{LEM3.2} with $Q=K_i U_i.$ We consider the following cases.

 \noindent{\bf Case I:} $U_i=Q_i.$ Using Lemma \ref{LEM3.2}, we get
 { \begin{align*}
     \miint{K_i U_i}|v-(v)_{K_i U_i}|\, dz &\leq c\lambda^{2-p}_ir_i\miint{K_i U_i}\left(f^{p-1}+a(z)f^{q-1}+b(z)f^{s-1}\right)\, dz\\
     &\quad +c(r_i+\lambda^{2-p}_ir^2_i)\miint{K_iU_i}f\, dz.
 \end{align*}}
 Since $r_i \leq R_2$ and $p\geq 2,$ we have
 \begin{align}\label{EQQ3.11}
 c(r_i+\lambda^{2-p}_ir^2_i)\miint{K_iU_i}f\, dz \leq c (R_2) r_i \miint{K_iU_i} f\, dz.     
 \end{align}
 Plugging \eqref{EQQ3.11} in the above estimate, we get
 \begin{align}\label{Equation3.12}
 \miint{K_i U_i}|v-(v)_{K_i U_i}|\, dz &\leq c\lambda^{2-p}_ir_i\miint{K_i U_i}\left(f^{p-1}+a(z)f^{q-1}+b(z)f^{s-1}\right)\, dz \nonumber\\
 &\qquad + cr_i \miint{K_i U_i}f\, dz.    
 \end{align}
 Since $U_i=Q_i,$ using $K^2\lambda^p_i\geq a(z_i)\lambda^q_i$ and $K^2\lambda^p_i\geq b(z_i)\lambda^s_i,$ we estimate 
 \begin{align*}
 &\miint{K_iU_i}a(z)f^{q-1}+b(z)f^{s-1}\, dz\\
 &\leq \miint{K_iU_i}|a(z)-a(z_i)|f^{q-1}\, dz +a(z_i)\miint{K_iU_i}f^{q-1}\,dz\\
 &\quad +\miint{K_iU_i}|b(z)-b(z_i)|f^{s-1}\, dz +b(z_i)\miint{K_iU_i}f^{s-1}\,dz\\
 &\leq [a]_{\alpha}(K_ir_i)^{\alpha}\miint{K_iU_i}f^{q-1}\,dz +K^2 \lambda^{p-q}_i\miint{K_iU_i}f^{q-1}\, dz\\
 &\quad +[b]_{\beta}(K_ir_i)^{\beta}\miint{K_iU_i}f^{s-1}\,dz +K^2 \lambda^{p-s}_i\miint{K_iU_i}f^{s-1}\, dz.
 \end{align*}
 Using H\"{o}lder's inequality and applying \eqref{EQQ3.2},  we further estimate
 { \begin{align*}
 \miint{K_iU_i}a(z)f^{q-1}+b(z)f^{s-1}\, dz &\leq [a]_{\alpha}(K_ir_i)^{\alpha}\left(\miint{K_iU_i}f^{s-1}\,dz\right)^{\frac{q-1}{s-1}}\\
 &\quad+K^2\lambda^{p-q}_i \left(\miint{K_iU_i}f^{s-1}\,dz\right)^{\frac{q-1}{s-1}}\\
 &\quad +[b]_{\beta}(K_ir_i)^{\beta}\lambda^{s-1}_i+K^2 \lambda^{p-s}_i\lambda^{s-1}_i\\
 &\leq [a]_{\alpha}(K_ir_i)^{\alpha}\lambda^{q-1}_i+[b]_{\beta}(K_ir_i)^{\beta} \lambda^{s-1}_i+2K^2 \lambda^{p-1}_i\\
 &\leq c\lambda^{p-1}_i,
 \end{align*}
 where the last inequality follows from Lemma \ref{LEM2.9} \ref{UUi}.}
 Plugging the above estimate in \eqref{Equation3.12}, we obtain
 \begin{align*}
\miint{K_i U_i}|v-(v)_{K_i U_i}|\, dz \leq c(\operatorname{data})\lambda_i r_i.    
 \end{align*}

 \noindent{\bf Case II:} $U_i=Q^q_i.$ Note that, in this case we have $K^2 \lambda^p_i<a(z_i)\lambda^q_i$ and $K^2 \lambda^p_i\geq b(z_i)\lambda^s_i.$ Using Lemma \ref{LEM3.2}, we get
{ \begin{align*}
     \miint{K_i U_i}|v-(v)_{K_i U_i}|\, dz &\leq c\underbrace{\frac{\lambda^{2}_ir_i}{g_q(z_i,\lambda_i)}\miint{K_i U_i}\left(f^{p-1}+a(z)f^{q-1}+b(z)f^{s-1}\right)\, dz}_{J_1}\\
     &\quad +\underbrace{c\left(r_i+\frac{\lambda^{2}_ir^2_i}{g_q(z_i,\lambda_i)}\right)\miint{K_iU_i}f\, dz}_{J_2}.
 \end{align*}
 Note that $J_2$ can be estimated as previous since $g_q(z_i,\lambda_i)>\lambda^p_i.$} {Now we estimate $J_1.$} First we note that, from Lemma \ref{lem : comparision of a(cdot) in p,q-phase} we have $\frac{a(z_i)}{2}\leq a(z)\leq 2a(z_i)$ and $[b]_{\beta}(50 K r_i)^{\beta}\lambda^s_i\leq (K^2-1)\lambda^p_i.$ Hence we have
 { \begin{align*}
    J_1 &\leq \frac{\lambda^2_i r_i}{\lambda^p_i}\miint{K_i U_i}f^{p-1}\,dz+ \frac{\lambda^2_i r_i}{a(z_i)\lambda^q_i}\miint{K_i U_i}a(z)f^{q-1}\,dz+\frac{\lambda^2_i r_i}{\lambda^p_i}\miint{K_i U_i}b(z)f^{s-1}\,dz\\
    &\leq c\lambda_i r_i+ \frac{2 \lambda^2_i r_i}{\lambda^q_i}\miint{K_i U_i}f^{q-1}\,dz\\
    &\quad+\frac{\lambda^2_ir_i}{\lambda^p_i}\left([b]_{\beta}(K_ir_i)^{\beta}\miint{K_iU_i}f^{s-1}\,dz+K^2 \lambda^{p-s}_i\miint{K_iU_i}f^{s-1}\, dz\right)\\
    &\leq c \lambda_i r_i.
 \end{align*}}
 The last part of the above calculation follows from Case I.

 \noindent{\bf Case III:} $U_i=Q^s_i.$ Note that, in this case we have $K^2 \lambda^p_i\geq a(z_i)\lambda^q_i$ and $K^2 \lambda^p_i< b(z_i)\lambda^s_i.$ Using Lemma \ref{LEM3.2}, we get
{\begin{align*}
     \miint{K_i U_i}|v-(v)_{K_i U_i}|\, dz &\leq c\underbrace{\frac{\lambda^{2}_ir_i}{g_s(z_i,\lambda_i)}\miint{K_i U_i}\left(f^{p-1}+a(z)f^{q-1}+b(z)f^{s-1}\right)\, dz}_{J_3}\\
     &\quad+\underbrace{c\left(r_i+\frac{\lambda^{2}_ir^2_i}{g_s(z_i,\lambda_i)}\right)\miint{K_iU_i}f\, dz}_{J_4}.
 \end{align*}}
 { The estimates of $J_3$ and $J_4$ are similar to case II.} We use the H\"{o}lder regularity of $a(z)$ and bounds on $b(z),$ i.e, $\frac{b(z_i)}{2}\leq b(z)\leq 2b(z_i),$ and we can arrive at the same conclusion.

 \noindent{\bf Case IV:} $U_i=Q^{q, s}_i.$ In this case, we have $K^2 \lambda^p_i<a(z_i)\lambda^q_i$ and $K^2 \lambda^p_i< b(z_i)\lambda^s_i.$ and again using Lemma \ref{LEM3.2}, we get
 { \begin{align*}
     &\miint{K_i U_i}|v-(v)_{K_i U_i}|\, dz\\
     &\qquad\leq c\underbrace{\frac{\lambda^{2}_ir_i}{g_{q,s}(z_i,\lambda_i)}\miint{K_i U_i}\left(f^{p-1}+a(z)f^{q-1}+b(z)f^{s-1}\right)\, dz}_{J_5}\\
     &\qquad \qquad +\underbrace{c\left(r_i+\frac{\lambda^{2}_ir^2_i}{g_{q,s}(z_i,\lambda_i)}\right)\miint{K_iU_i}f\, dz}_{J_6}.
 \end{align*}}
 From Lemma \ref{LEM2.9}, \ref{viii}, we have $\frac{a(z_i)}{2}\leq a(z)\leq 2a(z_i)$ and  $\frac{b(z_i)}{2}\leq b(z)\leq 2b(z_i)$ for every $z \in 200K Q_{r_i}(z_i).$ {The estimate of $J_6$ is same as previous as $\lambda^p_i+a(z_i)\lambda^q_i+b(z_i)\lambda^s_i>\lambda^p_i.$} { Next, we estimate $J_5.$} Then we have
 {\begin{align*}
     J_5 &\leq \frac{\lambda^2_i r_i}{\lambda^p_i}\miint{K_i U_i}f^{p-1}\,dz+ \frac{\lambda^2_i r_i}{a(z_i)\lambda^q_i}\miint{K_i U_i}a(z)f^{q-1}\,dz+\frac{\lambda^2_i r_i}{b(z_i)\lambda^s_i}\miint{K_i U_i}b(z)f^{s-1}\,dz\\
     &\leq \frac{\lambda^2_i r_i}{\lambda^p_i}\miint{K_i U_i}f^{p-1}\,dz+ \frac{2\lambda^2_i r_i}{\lambda^q_i}\miint{K_i U_i}f^{q-1}\,dz+\frac{2\lambda^2_i r_i}{\lambda^s_i}\miint{K_i U_i}f^{s-1}\,dz\\
     &\leq c\lambda_i r_i,
 \end{align*}
 where the last inequality follows from \eqref{EQQ3.2}.}
 This completes the proof of \eqref{EQQ3.8}. 

 Now we prove \eqref{EQQ3.10}. If $K_i U_i \not\subset U_{R_2, S_2}(z_0),$ then either $K_iB_i \subset B_{R_2}(x_0)$ and $K_i I_i \cap l^c_{S_2}(t_0)\neq \emptyset,$ or $K_iB_i \cap B^c_{R_2}\neq \emptyset.$ In the first case, we may apply \eqref{EQQ3.6} of Lemma \ref{LEM3.2}. Note that, since the right-hand side of \eqref{EQQ3.5} and \eqref{EQQ3.6} are the same, the proof follows from the previous case. In the second case, we apply Lemma \ref{lem: boundary poincare} with $\sigma =1$ and $\eta =d$ since $v(\cdot, t)\in W^{1,d}_0(B_{R_2}(x_0), \mathbb{R}^N).$ Indeed, we have
 \begin{align*}
     \miint{K_i U_i} |v|\, dz \leq cr_i \left(\miint{4K_i U_i}|\nabla v|^d\, dz\right)^{1/d}
 \end{align*}
 and using $|\nabla v|\leq cf,$ and \eqref{EQQ3.2} we obtain
 \begin{align*}
  \left(\miint{4K_i U_i}|\nabla v|^d\, dz\right)^{1/d} \leq c\left(\miint{4K_i U_i} f^d\, dz\right)^{1/d} \leq c\lambda_i,
 \end{align*}
 which completes the proof of \eqref{EQQ3.10}.

 Next, we focus on the estimates involving Steklov averages, i.e. \eqref{EQQ3.7} and \eqref{EQQ3.9}. If $K_i U_i \subset U_{R_2, S_2}(z_0),$ then from Lemma \ref{LEM3.2}, we get
 \begin{align*}
     \miint{K_i U_i}|v_h-(v_h)_{K_i U_i}|\, dz &\leq c \frac{|I_i|}{r_i}\miint{[K_iU_i]_h}\left(f^{p-1}+a(z)f^{q-1}+b(z)f^{s-1}\right)\,dz\\
     &\quad +c\left(|I_i|+r_i\right)\miint{[K_iU_i]_h}f\,dz.
 \end{align*}
 We note that $|I_i|\leq r^2_i \leq r_i R_2$ and $a(z)\leq ||a||_{\infty}, b(z)\leq ||b||_{\infty}.$ Moreover, we observe that for small $h>0,$ $[K_i U_i]_h$ may intersect $E(\Lambda),$ i.e, $[K_i U_i]_h\cap E(\Lambda)\neq \emptyset.$
 Using these observations and applying \eqref{EQQ3.1}, we obtain
 \begin{align*}
 \miint{K_i U_i}|v_h-(v_h)_{K_i U_i}|\, dz &\leq cr_i \miint{[K_i U_i]_h}f^{p-1}\,dz+cr_i||a||_{\infty}\miint{[K_iU_i]_h}f^{q-1}\,dz\\
 &\qquad+cr_i||b||_{\infty}\miint{[K_iU_i]_h}f^{s-1}\,dz\\
 &\leq cr_i \Lambda^{\frac{p-1}{p}}+cr_i ||a||_{\infty}\left(\miint{[K_iU_i]_h}f^{s-1}\,dz\right)^{\frac{q-1}{s-1}}\\
 &\qquad+cr_i||b||_{\infty}\Lambda^{\frac{s-1}{p}}\\
 &\leq cr_i\left(\Lambda^{\frac{p-1}{p}}+||a||_{\infty}\Lambda^{\frac{q-1}{p}}+||b||_{\infty}\Lambda^{\frac{s-1}{p}}\right)=c(\operatorname{data},\Lambda)r_i.
 \end{align*}
 This completes the proof of \eqref{EQQ3.7}. To prove \eqref{EQQ3.9} in the case $K_i B_i \subset B_{R_2}(x_0)$ and $K_i I_i \cap l^c_{S_2}\neq \emptyset,$ we need to estimate \eqref{EQQ3.6} which is the same as estimating \eqref{EQQ3.5}. When $K_i B_i \cap B_{R_2}\neq \emptyset,$ we may apply Lemma \ref{lem: boundary poincare} to get
 \begin{align*}
     \miint{K_i U_i}|v_h|\,dz \leq cr_i \left(\miint{4K_iU_i}|\nabla v_h|^d\,dz\right)^{\frac{1}{d}}.
 \end{align*}
 Now using Lemma \ref{lem: steklov average} and \eqref{EQQ3.1}, we estimate the right-hand side of the above expression as
 \begin{align*}
     \left(\miint{4K_i U_i}|\nabla v_h|^d\, dz\right)^{\frac{1}{d}}\leq c\left(\miint{[4K_iU_i]_h}|\nabla v|^d 
     { dz}\right)^{\frac{1}{d}}\leq c\left(\miint{[4K_iU_i]_h}f^d\,dz\right)^{\frac{1}{d}}\leq c\Lambda^{\frac{1}{p}},
 \end{align*}
 which gives the proof of \eqref{EQQ3.9}.
\end{proof}
\begin{corollary}\label{cor3.5}
We have the following estimates on $v^i_h, v^j_h$ and $v^i, v^j$ for every $i \in \mathbb{N}$ and $j \in \mathcal{I}:$
\begin{enumerate}[label=(\roman*),series=theoremconditions]
\item\label{cor: i} \begin{align}\label{eq:bound of dashint v-v^i}
    \miint{\frac{2}{K}U_i}|v-v^i|\,dz \leq c(\operatorname{data})\lambda_i r_i,
\end{align}
\item\label{cor:ii} \begin{align}
    \miint{\frac{2}{K}U_i}|v_h-v^i_h|\,dz \leq c(\operatorname{data}, \Lambda) r_i,
\end{align}
\item\label{cor: iii} \begin{align}
    |v^i_h-v^j_h|\leq c(\operatorname{data}, \Lambda) r_i,
\end{align}
\item \label{cor: iv} \begin{align}
    |v^i-v^j|\leq c(\operatorname{data}) \lambda_i r_i.
\end{align}
\end{enumerate}
\end{corollary}
\begin{proof}
We start by proving \ref{cor: i}. Indeed, we obtain
\begin{align*}
    \miint{\frac{2}{K}U_i}|v-v^i|\,dz &\leq \miint{\frac{2}{K}U_i}|v-(v)_{K_i U_i}|\, dz+\miint{\frac{2}{K}U_i}|(v)_{K_i U_i}-(v)_{\frac{2}{K}U_i}|\, dz\\
&\leq c\miint{K_i U_i}|v-(v)_{K_i U_i}|\,dz.
\end{align*}
If $K_i U_i \subset U_{R_2, S_2}(z_0),$ then from \eqref{EQQ3.8} we get
\begin{align*}
\miint{\frac{2}{K}U_i}|v-v^i|\,dz \leq  c\miint{K_i U_i}|v-(v)_{K_i U_i}|\,dz \leq c(\operatorname{data})\lambda_i r_i.   
\end{align*}
If $K_i U_i \not \subset U_{R_2, S_2}(z_0),$ from \eqref{EQQ3.10} we have
\begin{align*}
    \miint{\frac{2}{K}U_i}|v-v^i|\,dz\leq 2\miint{\frac{2}{K}U_i}|v|\,dz\leq c \miint{K_iU_i}|v|\,dz \leq c \lambda_i r_i.
\end{align*}
The proof of \ref{cor:ii} follows from \eqref{EQQ3.7} and \eqref{EQQ3.9} as above.
Now, let us prove \ref{cor: iv} and the proof of \ref {cor: iii} is similar. First, we assume $K_iU_i, K_jU_j \subset U_{R_2, S_2}(z_0).$ Using Lemma \ref{LEM2.9} \ref{Uij} and \ref{vii}, we estimate
{ \begin{align*}
    |v^i-v^j|&\leq |v^i-(v)_{K_iU_i}|+|v^j-(v)_{K_i U_i}|\\
    &=\left|\miint{\frac{2}{K}U_i}v-(v)_{K_i U_i}\,dz\right|+\left|\miint{\frac{2}{K}U_j}v-(v)_{K_i U_i}\,dz\right|\\
    &\leq c\miint{K_i U_i}|v-(v)_{K_i U_i}|\, dz \leq c\lambda_i r_i,
\end{align*}
where the last inequality follows from \eqref{EQQ3.8}.}
On the other hand, let $K_i U_i \not \subset U_{R_2, S_2}(z_0).$ In this case, again using Lemma \ref{LEM2.9} \ref{Uij} and \ref{vii}, we obtain
\begin{align*} 
|v^i-v^j|\leq \miint{\frac{2}{K}U_i}|v|\,dz + \miint{\frac{2}{K}U_j}|v|\,dz \leq c \miint{K_i U_i}|v|\,dz\leq c\lambda_i r_i.
\end{align*}
The other case $K_j U_j \not\subset U_{R_2, S_2}(z_0)$ follows from the fact that $r_i, r_j$ and $\lambda_i, \lambda_j$ are comparable.
\end{proof}

\subsection{Bounds on Lipschitz truncation and its derivatives}
In this subsection, we show that $v_h^\Lambda$, $v^\Lambda$ and their gradients are bounded.
\begin{lemma}\label{lem: bounds on Lipschitz truncation}
We have $|v^{\Lambda}_h(z)|\leq c(\operatorname{data}, \Lambda)$ and $|v^{\Lambda}(z)|\leq c(\operatorname{data})\lambda_i$ for every $z\in U_i.$
\end{lemma}
\begin{proof}
    Fix $z\in U_i$. Then, for each $j\in\mathcal{I}$, we obtain from Lemma \ref{LEM2.9}\ref{vii}, Lemma \ref{lem : miint f^gamma bound} and the definition of $v^j$ that
    \begin{equation*}   
        |v^j|\leq \miint{\frac{2}{K}U_j} |v| \,dz \leq c\miint{K_i U_i} f \, dz\leq c \lambda_i.
    \end{equation*}
    Moreover, we have from Lemma \ref{lem : partition of unity} that
    \begin{equation}    \label{eq : definiton of v^Lambda with repect to I}
        v^\Lambda(z)=\sum_{j\in\mathbb{N}} v^j\omega_j(z)=\sum_{j\in \mathcal{I}}v^j\omega_j(z).
    \end{equation}
    Therefore, we conclude from Lemma \ref{LEM2.9}\ref{ix} and Lemma \ref{lem : partition of unity} that
    $$
    |v^\Lambda(z)|\leq \sum_{j\in \mathcal{I}}|v^j||\omega_j(z)|\leq \sum_{j\in \mathcal{I}}|v^j|\leq c\lambda_i.
    $$
    On the other hand, for each $j\in \mathcal{I}$, we get from Lemma \ref{lem: steklov average} and Lemma \ref{lem : miint f^gamma bound} that
    $$
    |v_h^i|\leq \miint{\frac{2}{K}U_i} |v_h(z)|\, dz\leq c\miint{\frac{2}{K}[U_i]_h} f\, dz\leq c(\operatorname{data},\Lambda).
    $$
    As in \eqref{eq : definiton of v^Lambda with repect to I}, we have 
    \begin{equation}    \label{eq : definiton of v_h^Lambda with repect to I}
        v_h^\Lambda(z)=\sum_{j\in \mathcal{I}}v_h^j\omega_j(z),
    \end{equation}
    and hence we have the conclusion.
\end{proof}
\begin{lemma}\label{lem:  bounds on derivaties of Lipschitz truncation}
    For any $z\in U_i,$ we have 
    \begin{align}\label{eq : bound of gradient v_h^Lambda}
        |\nabla v^{\Lambda}_h(z)|\leq c(\operatorname{data}, \Lambda)\,\,\,\, \text{and}\,\,\,\, |\nabla v^{\Lambda}(z)|\leq c(\operatorname{data})\lambda_i.
    \end{align}
    Furthermore, we have 
    \begin{align}\label{EQQ3.18}
        |\partial_t v^{\Lambda}_h(z)|\leq c(\operatorname{data}, \Lambda)r^{-1}_i
    \end{align}
    and 
    { \begin{align}
        |\partial_t v^{\Lambda}(z)| \leq \begin{cases}
            c(\operatorname{data})r^{-1}_i\lambda^{p-1}_i,\,\,\, &\text{if}\,\,\, U_i=Q_i,\\
            c(\operatorname{data})r^{-1}_ig_{q}(z_i,\lambda_i)\lambda^{-1}_i,\,\,\, &\text{if}\,\,\, U_i=Q^q_i,\\
             c(\operatorname{data})r^{-1}_ig_{s}(z_i,\lambda_i)\lambda^{-1}_i,\,\,\, &\text{if}\,\,\, U_i=Q^s_i,\\
              c(\operatorname{data})r^{-1}_i\Lambda\lambda^{-1}_i,\,\,\, &\text{if}\,\,\, U_i=Q^{q,s}_i.\\
        \end{cases}
    \end{align}}
\end{lemma}
\begin{proof}
    Fix $z\in U_i$. Then, by \eqref{eq : definiton of v^Lambda with repect to I} and \eqref{eq : definiton of v_h^Lambda with repect to I}, we get
    $$
    \nabla v_h^\Lambda (z)=\nabla \left(\sum_{j\in\mathcal{I}} v_h^j\omega_j(z)\right)=\sum_{j\in\mathcal{I}} v_h^j\nabla\omega_j(z)
    $$
    and
    $$
    \nabla v^\Lambda (z)=\sum_{j\in\mathcal{I}} v^j\nabla\omega_j(z).
    $$
    Note that Lemma \ref{lem : partition of unity} implies 
    $$
    0=\nabla\left(\sum_{j\in\mathbb{N}}\omega_j(z)\right)=\nabla\left(\sum_{j\in\mathcal{I}}\omega_j(z)\right)=\sum_{j\in\mathcal{I}}\nabla\omega_j(z).
    $$
    Thus, we get from Lemma \ref{LEM2.9} \ref{ix}, Lemma \ref{lem : partition of unity} and Corollary \ref{cor3.5} that 
    $$
    \begin{aligned}
        |\nabla v_h^\Lambda(z)|=\left|\sum_{j\in\mathcal{I}}(v_h^j-v_h^i)\nabla\omega_j(z)\right|\leq \sum_{j\in\mathcal{I}}|v_h^j-v_h^i||\nabla \omega_j|\leq c(\operatorname{data},\Lambda)
    \end{aligned}
    $$
    and
    $$
    |\nabla v^\Lambda(z)|\leq\sum_{j\in\mathcal{I}}|v^j-v^i||\nabla \omega_j| \leq c(\operatorname{data})\lambda_i.
    $$
    Next, by the above arguments, we get
    $$
    \partial_t v_h^\Lambda (z)=\sum_{j\in\mathcal{I}} v_h^j\partial_t\omega_j(z)=\sum_{j\in\mathcal{I}} (v_h^j-v_h^i)\partial_t\omega_j(z)
    $$
    and
    $$
    \partial_t v^\Lambda (z)=\sum_{j\in\mathcal{I}} v^j\partial_t\omega_j(z)=\sum_{j\in\mathcal{I}} (v^j-v^i)\partial_t\omega_j(z).
    $$
    Therefore, by Lemma \ref{LEM2.9} \ref{ix}, Lemma \ref{lem : partition of unity} and Corollary \ref{cor3.5}, we have the conclusion.
\end{proof}
\begin{lemma}\label{LEMMA4.9}
Let $E(\Lambda)$ is defined as in \eqref{defn of E}. Then $v^{\Lambda}_h, v^{\Lambda},$ the Lipschitz truncation defined in \eqref{lip trunc_h}, \eqref{test fn} satisfy the following estimates:
\begin{enumerate}[label=(\roman*),series=theoremconditions]
\item \label{lemma5.8 i}$$\iint_{E(\Lambda)^c}\left|v_h-v^{\Lambda}_h\right||\partial_t v^{\Lambda}_h|dz \leq c(\operatorname{data}, \Lambda)|E(\Lambda)^c|,$$
\item \label{lemma5.8 ii}$$\iint_{E(\Lambda)^c}\left|v-v^{\Lambda}\right||\partial_t v^{\Lambda}|dz \leq c(\operatorname{data}) \Lambda|E(\Lambda)^c|,$$
\item \label{lemma5.8 iii} \begin{align*}
H(z, |v^{\Lambda}|)+H(z, |\nabla v^{\Lambda}|)\leq c(\operatorname{data})\Lambda     
\end{align*} for almost every $z \in \mathbb{R}^{n+1}.$
\end{enumerate}
\end{lemma}
\begin{proof}
 We start with proving \ref{lemma5.8 i}:
\begin{align}\label{EQQ5.20}
\iint_{E(\Lambda)^c}|v_h-v^{\Lambda}_h||\partial_t v^{\Lambda}_h|dz &\leq \iint_{E(\Lambda)^c}\sum_{i\in \mathbb{N}}|v_h-v^i_h||\omega_i||\partial_t v^{\Lambda}_h|dz\\ &\leq \sum_{i\in \mathbb{N}}\iint_{\frac{2}{K}U_i}|v_h-v^i_h||\omega_i||\partial_t v^{\Lambda}_h|dz\nonumber\\
&\leq \sum_{i \in \mathbb{N}}||\partial_t v^{\Lambda}_h||_{L^{\infty}(\frac{2}{K}U_i)}\iint_{\frac{2}{K}U_i}|v_h-v^i_h|dz.
\end{align}  
Now using Corollary \ref{cor3.5} \ref{cor:ii} and \eqref{EQQ3.18}, we can estimate the last term of the above inequality to obtain
\begin{align*}
\iint_{E(\Lambda)^c}|v_h-v^{\Lambda}_h||\partial_t v^{\Lambda}_h|dz \leq c \sum_{i \in \mathbb{N}}|U_i|=c\sum_{i\in \mathbb{N}}\left|\frac{1}{6K^6}U_i\right|=c(\operatorname{data}, \Lambda)|E(\Lambda)^c|.
\end{align*}
The proof of \ref{lemma5.8 ii} can be obtained similarly. To prove \ref{lemma5.8 iii}, first let $z\in E(\Lambda).$ In this case,
{ \begin{align*}
    &|v^{\Lambda}(z)|^p+a(z)|v^{\Lambda}(z)|^q+b(z)|v^{\Lambda}(z)|^s+|\nabla v^{\Lambda}(z)|^p+a(z)|\nabla v^{\Lambda}(z)|^q+b(z)|\nabla v^{\Lambda}(z)|^s\\
    &\qquad=|v(z)|^p+a(z)|v(z)|^q+b(z)|v(z)|^s+|\nabla v(z)|^p+a(z)|\nabla v(z)|^q+b(z)|\nabla v(z)|^s\\
    &\qquad\leq \left(f^p+a(z)f^q+b(z)f^s\right)\leq c\Lambda.
\end{align*}}
Now let us consider $z\in E(\Lambda)^c.$ Then $z\in U_i$ for some $i \in \mathbb{N}.$ Using Lemma \ref{lem: bounds on Lipschitz truncation} and Lemma \ref{lem:  bounds on derivaties of Lipschitz truncation}, we get
{ \begin{align*}
&|v^{\Lambda}(z)|^p+a(z)|v^{\Lambda}(z)|^q+b(z)|v^{\Lambda}(z)|^s+|\nabla v^{\Lambda}(z)|^p+a(z)|\nabla v^{\Lambda}(z)|^q+b(z)|\nabla v^{\Lambda}(z)|^s\\
&\qquad\leq c(\lambda^p_i+a(z)\lambda^q_i+b(z)\lambda^s_i).    
\end{align*}}
\noindent{\bf Case I:} $U_i=Q_i.$ Using the H\"{o}lder regularity of $a(z)$ and $b(z)$, and Lemma \ref{LEM2.9}, \ref{UUi} we obtain
\begin{align*}
 \lambda^p_i+a(z)\lambda^q_i+b(z)\lambda^s_i &\leq \lambda^p_i+a(z_i)\lambda^q_i+[a]_{\alpha}r^{\alpha}_i\lambda^q_i+b(z_i)\lambda^s_i+[b]_{\beta}r^{\beta}_i\lambda^s_i\\
 &\leq c \left(\lambda^p_i+a(z_i)\lambda^q_i+b(z_i)\lambda^s_i\right)=c\Lambda.
\end{align*}
\noindent{\bf Case II:} $U_i=Q^q_i.$ In this case, we use Lemma \ref{lem : comparision of a(cdot) in p,q-phase} to get
\begin{align*}
\lambda^p_i+a(z)\lambda^q_i+b(z)\lambda^s_i &\leq \lambda^p_i+2a(z_i)\lambda^q_i+b(z_i)\lambda^s_i+[b]_{\beta}r^{\beta}_i\lambda^s_i\\
&\leq c(\lambda^p_i+a(z_i)\lambda^q_i+b(z_i)\lambda^s_i)=c\Lambda.
\end{align*}
\noindent{\bf Case III:} $U_i=Q^s_i.$ In this case, we use Lemma \ref{lem : comparision of b(cdot) in p,s-phase} to get
\begin{align*}
\lambda^p_i+a(z)\lambda^q_i+b(z)\lambda^s_i &\leq \lambda^p_i+a(z_i)\lambda^q_i+2b(z_i)\lambda^s_i+[a]_{\alpha}r^{\alpha}_i\lambda^q_i\\
&\leq c(\lambda^p_i+a(z_i)\lambda^q_i+b(z_i)\lambda^s_i)=c\Lambda.
\end{align*}
\noindent{\bf Case IV:} $U_i=Q^{q,s}_i.$ In this case we use Lemma \ref{LEM2.1} and obtain
\begin{align*}
\lambda^p_i+a(z)\lambda^q_i+b(z)\lambda^s_i \leq \lambda^p_i+2a(z_i)\lambda^q_i+2b(z_i)\lambda^s_i \leq 2 \Lambda.   
\end{align*}
This completes the proof.
\end{proof}
\subsection{Lipschitz regularity of $v^{\Lambda}_h$} In this section, we show that $v^{\Lambda}_h$ is Lipschitz continuous with respect to the metric
\begin{align*}
    d_{\lambda^p}(z, w)=\max\left\{|x-y|, \sqrt{\lambda^{p-2}|t-s|}\right\},
\end{align*}
where $z, w \in U_{R_2, S_2}(z_0)$ with $z=(x,t)$ and $w=(y,s)$ and $\lambda$ is chosen such that $\Lambda=\lambda^p+||a||_{\infty}\lambda^q+||b||_{\infty}\lambda^s.$ { Let us recall the definition of $Q_{l,\lambda}(w),$ that is,
 \begin{align*}
 Q_{l, \lambda}(w):=B_{l}(y)\times(s-\lambda^{2-p}l^2, s+\lambda^{2-p}l^2).
 \end{align*}
\begin{lemma}[Campanato characterization]\label{capmanto}
    Assume that $f\in L^1_{loc}(\mathbb{R}^{n+1}).$ Then there exist a constant $c=c(n)$ and a set $E\subset \mathbb{R}^{n+1}$ with $|E|=0$ such that 
\begin{align*}
    |f(z)-f(w)| & \leq c(n)\,d_{\lambda^p}(z,w)\sup_{l>0}\miint{Q_{l, \lambda}(w)}\frac{\left|f(\tilde{w})-(f)_{Q_{l, \lambda}(w)}\right|}{l} d\tilde{w}\\
 &\qquad+ c(n)\,d_{\lambda^p}(z,w)\sup_{l>0}\miint{Q_{l, \lambda}(z)}\frac{\left|f(\tilde{z})-(f)_{Q_{l, \lambda}(z)}\right|}{l} d\tilde{z} .
    \end{align*}
    for every $z,w \in \mathbb{R}^{n+1}\setminus E.$
\end{lemma}
\begin{proof}
    Since $Q_{d_{\lambda^p}(z, w),\lambda}(z)\subset Q_{2 d_{\lambda^p}(z, w),\lambda}(w)$, by replacing $B(x,r)$ with $Q_{l,\lambda}(z)$ and $|x-y|$ with $d_{\lambda^p}(z,w),$ and taking $\beta=1$ in the proof of \cite[Lemma 4.13]{Kinnunen2021}, we obtain that 
    \begin{align*}
    |f(z)-f(w)|&\leq c(n)\,d_{\lambda^p}(z,w)\sup_{0<l< 4d_{\lambda^p}(z,w)}\miint{Q_{l, \lambda}(w)}\frac{\left|f(\tilde{w})-(f)_{Q_{l, \lambda}(w)}\right|}{l} d\tilde{w}\\
 &\qquad+ c(n)\,d_{\lambda^p}(z,w)\sup_{0<l<4d_{\lambda^p}(z,w)}\miint{Q_{l, \lambda}(z)}\frac{\left|f(\tilde{z})-(f)_{Q_{l, \lambda}(z)}\right|}{l} d\tilde{z}\\
 &\leq c(n)\,d_{\lambda^p}(z,w)\sup_{l>0}\miint{Q_{l, \lambda}(w)}\frac{\left|f(\tilde{w})-(f)_{Q_{l, \lambda}(w)}\right|}{l} d\tilde{w}\\
 &\qquad+ c(n)\,d_{\lambda^p}(z,w)\sup_{l>0}\miint{Q_{l, \lambda}(z)}\frac{\left|f(\tilde{z})-(f)_{Q_{l, \lambda}(z)}\right|}{l} d\tilde{z}.
\end{align*}
This completes the proof.
\end{proof}
We remark that the above conclusion holds for every $z,\, w\in \mr^{n+1}$ by \cite[Remark 4.14]{Kinnunen2021}.}
Now we are ready to prove the Lipschitz regularity of $v^{\Lambda}_h.$
\begin{lemma}\label{LIP LEM}
    There exists a constant $c_{\Lambda}=c(\operatorname{data}, \Lambda)$ such that
    \begin{align*}
        |v^{\Lambda}_h(z)-v^{\Lambda}_h(w)|\leq c_{\Lambda}d_{\lambda^p}(z, w)
    \end{align*}
    for every $z, w \in \mathbb{R}^{n+1}.$
\end{lemma}
\begin{proof} 
 We first note that, since $\lambda_i\ge \lambda$ from the definition of $d,$ we have 
 \begin{align}\label{EQQ5.23}
     d_{\lambda^p}(z,w) \leq d_{i}(z,w)\,\,\,\,\, \text{for all}\,\,\, i\in \mathbb{N}.
 \end{align}
{Applying Lemma \ref{capmanto}, we get
 \begin{equation} \label{capmanto char}
 \begin{aligned}
 |v^{\Lambda}_h(z)-v^{\Lambda}_h(w)| &\leq c(n)d_{\lambda^p}(z,w)\sup_{l>0}\miint{Q_{l, \lambda}(w)}\frac{\left|v^{\Lambda}_h(\tilde{w})-(v^{\Lambda}_h)_{Q_{l, \lambda}(w)}\right|}{l} d\tilde{w}\\
 &\quad + c(n)d_{\lambda^p}(z,w)\sup_{l>0}\miint{Q_{l, \lambda}(z)}\frac{\left|v^{\Lambda}_h(\tilde{z})-(v^{\Lambda}_h)_{Q_{l, \lambda}(z)}\right|}{l} d\tilde{z} .
 \end{aligned}
 \end{equation}}
Note that $v^{\Lambda}_h \in L^1_{loc}(\mathbb{R}^{n+1}).$ Therefore from \eqref{capmanto char}, it is enough to show that there exists a constant $c_\Lambda=c_\Lambda(\operatorname{data}, \Lambda)$ such that
{ \begin{align}\label{EQQ5.25}
\miint{Q_{l, \lambda}(w)}\frac{|v^{\Lambda}_h(\tilde{w})-(v^{\Lambda}_h)_{Q_{l, \lambda}(w)}| }{l}d\tilde{w}\leq c_\Lambda\,\,\,\,\, \text{for all}\,\,\, Q_{l, \lambda}(w)\subset \mathbb{R}^{n+1}. 
\end{align}}
We fix the cube $Q_{l, \lambda}(w)$ and prove the above estimate \eqref{EQQ5.25} when $2Q_{l, \lambda}(w)$ completely lies on the bad set $E(\Lambda)^c,$ or $Q_{l, \lambda}(w)$ lies inside the bad set $E(\Lambda)^c,$ but $2Q_{l, \lambda}(w)$ meets the good set $E(\Lambda)$ or both $Q_{l, \lambda}(w)$ and $2Q_{l, \lambda}(w)$ meets the good set $E(\Lambda).$ 

\textit{Case 1: $2Q_{l,\lambda}(\omega) \subset E(\Lambda)^c$.} Let $z\in Q_{l,\lambda}(\omega)$. Then, by Lemma \ref{LEM2.9}\ref{i} there exists $i\in \mathbb{N}$ such that $z\in U_i$. Using \eqref{EQQ5.23}, we have
$$
l\leq d_{\lambda^p}(z,E(\Lambda))\leq d_{\lambda^p}(z,z_i)+d_{\lambda^p}(z_i,E(\Lambda))\leq d_i(z,z_i)+d_i(z_i,E(\Lambda)).
$$
Here, $z_i$ is the centre of $U_i$. Since $U_i$ is a ball of radius $r_i$ with respect to the metric $d_i(\cdot,\cdot)$, we get from Lemma \ref{LEM2.9}\ref{iv} that
$$
l\leq d_i(z,z_i)+d_i(z_i,E(\Lambda))\leq r_i+5r_i=6r_i.
$$
Lemma \ref{lem:  bounds on derivaties of Lipschitz truncation} implies the uniform estimate
\begin{equation}\label{eq : bound of partial_t v_h^Lambda}
    |\partial_t v_h^\Lambda(z)|\leq c(\operatorname{data},\Lambda)r_i^{-1}\leq c(\operatorname{data},\Lambda)l^{-1}
\end{equation}
for all $z\in Q_{l,\lambda}(w)$.

To prove \eqref{EQQ5.25}, let $z_1,z_2\in Q_{l,\lambda}(w)$ with $z_1=(x_1,t_1)$ and $z_2=(x_2,t_2)$. Since $v_h^\Lambda\in C^\infty(E(\Lambda)^c,\mr^N)$ and $Q_{l,\lambda}(w)\subset E(\Lambda)^c$, by the intermediate value theorem, \eqref{eq : bound of gradient v_h^Lambda} and \eqref{eq : bound of partial_t v_h^Lambda}, we have
$$
\begin{aligned}
    |v_h^\Lambda(z_1)-v_h^\Lambda(z_2)|&\leq |v_h^\Lambda(x_1,t_1)-v_h^\Lambda(x_2,t_1)|+|v_h^\Lambda(x_2,t_1)-v_h^\Lambda(x_2,t_2)|\\
    &\leq cl\sup_{z\in Q_{l,\lambda}(\omega)} |\nabla v_h^\Lambda(z)|+ c\lambda^{2-p}l^2 \sup_{z\in Q_{l,\lambda}(\omega)} |\partial_t v_h^\Lambda(z)|\leq c(\operatorname{data},\Lambda)l.
\end{aligned}
$$
Thus, we conclude
{ $$
\begin{aligned}
    \miint{Q_{l, \lambda}(w)}\frac{|v^{\Lambda}_h-(v^{\Lambda}_h)_{Q_{l, \lambda}(w)}| }{l}d\tilde{w}&\leq \miint{Q_{l, \lambda}(w)}\miint{Q_{l, \lambda}(w)}\frac{|v^{\Lambda}_h(z_1)-v^{\Lambda}_h(z_2)| }{l}dz_2 dz_1\\
    &\leq c_\Lambda(\operatorname{data,\Lambda}).
\end{aligned}
$$}

\textit{Case 2: $2Q_{l,\lambda}(w)\cap E(\Lambda)\neq \emptyset$ and $Q_{l,\lambda}(w)\cap E(\Lambda)^c=\emptyset$.} In this case $v_h^\Lambda=v_h$ and hence
{ $$
\miint{Q_{l,\lambda}(\omega)}\frac{|v_h^\Lambda-(v_h^\Lambda)_{Q_{l,\lambda}(w)}|}{l}\, d\tilde{w} = \miint{Q_{l,\lambda}(\omega)}\frac{|v_h-(v_h)_{Q_{l,\lambda}(w)}|}{l}\, d\tilde{w}.
$$}
We denote $w=(y,s)$. If $B_l(y)\subset B_{R_2}(x_0)$, then $l\leq R_2$ is satisfied and Lemma \ref{lem : miint f^gamma bound} and Lemma \ref{LEM3.2} imply that 
{ $$
\miint{Q_{l,\lambda}(w)} \frac{|v_h -(v_h)_{Q_{l,\lambda}(w)}|}{l}\,d\tilde{w}\leq c(\operatorname{data},\Lambda)(\lambda^{2-p}+\lambda^{2-p}l+1)\leq c(\operatorname{data},\Lambda).
$$}
On the other hand, if $B_l(y) \not\subset B_{R_2}(x_0)$, we apply Lemma \ref{lem: boundary poincare} with $\sigma =1$ and $\eta=d$ to get
{ $$
\miint{Q_{l,\lambda}(w)} \frac{|v_h -(v_h)_{Q_{l,\lambda}(w)}|}{l}\,d\tilde{w}\leq 2 \miint{Q_{l,\lambda}(w)} \frac{|v_h|}{l}\,d\tilde{w}\leq c \miint{4Q_{l,\lambda}(w)} |\nabla v_h|\,d\tilde{w}.
$$}
Recalling that in this case $2Q_{l,\lambda}(w)\cap E(\Lambda)\neq \emptyset$, we conclude with Lemma \ref{lem : miint f^gamma bound} that
{ $$
\miint{4Q_{l,\lambda}(w)} |\nabla v_h|\,d\tilde{w}\leq c\miint{[4Q_{l,\lambda}(w)]_h} f\, d\tilde{w}\leq c(\operatorname{data},\Lambda).
$$}
\textit{Case 3: $2Q_{l,\lambda}(w)\cap E(\Lambda)\neq \emptyset$ and $Q_{l,\lambda}\cap E(\Lambda)\neq \emptyset$.} We define the index set
$$
P=\{i\in\mathbb{N}:Q_{l,\lambda}(w)\cap \frac{2}{K}U_i \neq \emptyset\}.
$$
We want to show that the radii $r_i$ are bounded uniformly by $l$ with respect to $i\in P.$ Let $i\in P$, $w_1\in Q_{l,\lambda}(w)\cap \frac{2}{K}U_i$ and $w_2\in 2Q_{l,\lambda}(w)\cap E(\Lambda)$ with $w_1=(y_1,s_1)$ and $w_2=(y_2,s_2)$. By Lemma \ref{LEM2.9} \ref{iii} and $w_1\in \frac{2}{K}U_i$, we obtain
$$
3r_1\leq d_i(U_i,E(\Lambda))\leq d_i(z_i,w_2)\leq d_i(z_i,w_1)+d_i(w_1,w_2)\leq 2r_i+d_i(w_1,w_2),
$$
and hence $r_i\leq d_i(w_1,w_2)$. Moreover, since $\lambda\leq \lambda_i$ and $w_1,w_2\in 2Q_{l,\lambda}(w)$, we get
$$
\begin{aligned}
    d_i(w_1,w_2)&\leq \max\left\{|y_1-y_2|,\sqrt{\Lambda\lambda_i^{-2}|s_1-s_2|}\right\}\\
    &\leq \Lambda^{\frac{1}{2}}\max\left\{|y_1-y_2|,\sqrt{\lambda^{p-2}|s_1-s_2|}\right\}\\
    &=\Lambda^{\frac{1}{2}} d_{\lambda^p}(w_1,w_2)\leq 4\Lambda^\frac{1}{2}l.
\end{aligned}
$$
Thus, we have $r_i\leq c(\Lambda)l$.

Note that
{ $$
\begin{aligned}
    \miint{Q_{l, \lambda}(w)}\frac{|v^{\Lambda}_h-(v^{\Lambda}_h)_{Q_{l, \lambda}(w)}| }{l}d\tilde{w}&\leq 2\miint{Q_{l, \lambda}(w)}\frac{|v^{\Lambda}_h-(v_h)_{Q_{l, \lambda}(w)}| }{l}d\tilde{w}\\
    &\leq 2\miint{Q_{l, \lambda}(w)}\frac{|v^{\Lambda}_h-v_h|}{l}d\tilde{w}\\
    &\quad +2\miint{Q_{l, \lambda}(w)}\frac{|v_h-(v_h)_{Q_{l, \lambda}(w)}| }{l}d\tilde{w}.
\end{aligned}
$$}
As in Case 2, we can estimate the second term on the right-hand side. To estimate the first term, we obtain from the definition of $v_h^\Lambda$ and the fact that $w_i$ is supported in $\frac{2}{K}U_i$ that
{ $$
\begin{aligned}
    \miint{Q_{l, \lambda}(w)}\frac{|v^{\Lambda}_h-v_h|}{l}d\tilde{w}&=\miint{Q_{l, \lambda}(w)}\frac{|\sum_{i\in P}(v_h-v_h^i)w_i|}{l}d\tilde{w}\\
    &\leq \miint{Q_{l, \lambda}(w)} \sum_{i\in P}\frac{|v_h-v_h^i|w_i}{l}d\tilde{w}\\
    &=\frac{1}{|Q_{l,\lambda}(w)|}\sum_{i\in P}\iint_{Q_{i,\lambda}(w)\cap \frac{2}{K}U_i}\frac{|v_h-v_h^i|w_i}{l}d\tilde{w}.
\end{aligned}
$$}
Since $w_i\leq 1$ and $r_i\leq c(\Lambda)l$, we have
{ $$
\frac{1}{|Q_{l,\lambda}(w)|}\sum_{i\in P}\iint_{Q_{i,\lambda}(w)\cap \frac{2}{K}U_i}\frac{|v_h-v_h^i|w_i}{l}d\tilde{w} \leq \frac{c(\Lambda)}{|Q_{l,\lambda}(w)|}\sum_{i\in P}\iint_{\frac{2}{K}U_i}\frac{|v_h-v_h^i|}{r_i}d\tilde{w}.
$$}
Combining the previous inequalities with \eqref{eq:bound of dashint v-v^i}, we obtain
{ $$
\miint{Q_{l, \lambda}(w)}\frac{|v^{\Lambda}_h-v_h|}{l}d\tilde{w}\leq \frac{c(\operatorname{data},\Lambda)}{|Q_{l,\lambda}(w)|}\sum_{i\in P} |U_i|.
$$}
Since $r_i\leq c(\Lambda)l$, we get $U_i\subset c(\operatorname{data},\Lambda)Q_{i,\Lambda}(w)$ for every $i\in P$. By Lemma \ref{LEM2.9}\ref{ii}, we have that
{ $$
\miint{Q_{l,\lambda}(w)} \frac{|v_h^\Lambda-v_h|}{l}\,d\tilde{w} \leq \frac{c(\operatorname{data},\Lambda)}{|Q_{l,\lambda}(w)|}\sum_{i\in P} \left|\frac{1}{6K^6}U_i\right|\leq c(\operatorname{data},\Lambda).
$$}
Thus, the proof is completed.
\end{proof}
\begin{corollary}\label{COR4.11}
    Let $E(\Lambda)$ be defined in \eqref{defn of E}. Then $v_h$ satisfies the estimate
    \begin{align*}
        H(z,|v_h^\Lambda(z)|)+H(z,|\nabla v_h^\Lambda (z)|)\leq c(\operatorname{data}, \Lambda) \,\,\,\, \text{for\,\, a.e.}\,\,\,\, z\in \mathbb{R}^{n+1}.
    \end{align*}
\end{corollary}
\begin{proof}
Since $v^{\Lambda}_h(z)=0$ in $U_{R_2, S_2}(z_0)^c$ and $v^{\Lambda}_h$ is Lipschitz continuous for almost every $z\in \mathbb{R}^{n+1}$, we get that $|v^{\Lambda}_h(z)|\leq c(\operatorname{data}, \Lambda)$ and $|\nabla v^{\Lambda}_h(z)|\leq c(\operatorname{data}, \Lambda)$  for almost every $z\in \mathbb{R}^{n+1}.$   Then we have
{\begin{align*}
    &|v^{\Lambda}_h(z)|^p+a(z)|v^{\Lambda}_h(z)|^q+b(z)|v^{\Lambda}_h(z)|^s+|\nabla v^{\Lambda}_h(z)|^p+a(z)|\nabla v^{\Lambda}_h(z)|^q+b(z)|\nabla v^{\Lambda}_h(z)|^s\\
    &\qquad\leq c(\operatorname{data}, \Lambda)(1+||a||_{\infty}+||b||_{\infty}).
\end{align*}}
This completes the proof.
\end{proof}
\subsection{Some more properties of Lipschitz truncation}In the following proposition, we collect some more properties of Lipschitz truncation.
\begin{proposition} \label{prop Lip_trunc}
    Let $E(\Lambda)$ be as in \eqref{defn of E}, and let $\eta, \zeta$ be the cut-off functions. Then $\{v_h^\Lambda \}_{h>0}$ and a function $v^\Lambda$ satisfy the following properties:
    \begin{enumerate}[label=(\roman*),series=theoremconditions]
    \item\label{p2-1} $v_h^\Lambda\in W_0^{1,2}(supp(\zeta);L^2(supp(\eta),\mathbb{R}^N))\cap L^\infty(supp(\zeta);W^{1,\infty}_{0}(supp(\eta),\mathbb{R}^N))$.
    \item\label{p2-2} $v^\Lambda\in L^\infty(supp(\zeta)+h_0;W^{1,\infty}_{0}(supp(\eta),\mathbb{R}^N))$.
    \item\label{p2-3} $v_h^\Lambda= v_h,$ $v^\Lambda= v$, $\nabla v_h^\Lambda= \nabla v_h,$ $\nabla v^\Lambda=\nabla v$ a.e. in $E(\Lambda)$.
     \item\label{p3-1} $v_h^\Lambda \to v^\Lambda$ in $L^\infty(\Omega,\mathbb{R}^N)$ as $h\to0^+$, taking a subsequence if necessary.
    \item\label{p3-2} $\nabla v_h^\Lambda\to \nabla v^\Lambda$ and $\partial_t v_h^\Lambda\to \partial_t v^\Lambda$ a.e. in $E(\Lambda)^c$ as $h\to0^+$. 
\end{enumerate}
\end{proposition}
\begin{proof}
To show \ref{p2-1}, we first note from Definition \ref{weak solution} and the definition of Steklov averages that $v_h \in W^{1,2}_0(supp(\zeta); L^2(supp(\eta), \mathbb{R}^N))$. Moreover, the definition of $v^{\Lambda}_h$ in \eqref{lip trunc_h} only matters for finite sum since $|\mathcal{I}|\leq c$. Therefore, together with Lemma \ref{LIP LEM}, we complete the proof of \ref{p2-1}.

The proof of \ref{p2-2} is obvious from Lemma \ref{LEMMA4.9}~\ref{lemma5.8 iii}. Indeed, we have $|v^{\Lambda}(z)|\leq c\Lambda^{1/p}$ and $|\nabla v^{\Lambda}(z)|\leq c\Lambda^{1/p}$ for all $z\in \mathbb{R}^{n+1}.$

The proof of \ref{p2-3} follows from the definitions \eqref{lip trunc_h} and \eqref{test fn}. In fact, since $\omega_i(z)=0$ for $z\in E(\Lambda),$ we have the proof.

Since $v^{\Lambda}_h(z)=0$ for $z\in U^c_{R_2, S_2}(z_0),$ using Lemma \ref{LIP LEM} we see that $\{v^{\Lambda}_h\}_{h>0}$ is equicontinuous and uniformly bounded. By the properties of Steklov averages, we already have $v^{\Lambda}_h \to v^{\Lambda}$ as $h \to 0+.$ Hence, by Arzela-Ascoli theorem, we get $v^{\Lambda}_h\to v^{\Lambda}$ in $L^{\infty}(\Omega, \mathbb{R}^N).$ This proves \ref{p3-1}.

The proof of \ref{p3-2} follows from $v^i_h \to v^i$ as $h \to 0$ and the expressions computed in Lemma \ref{lem:  bounds on derivaties of Lipschitz truncation}.
\end{proof}
\section{\bf Proof of energy estimate}  \label{sec : proof of main theorem} 
In this section, we prove the energy estimate in Theorem \ref{thm : the Caccioppoli inequality} using the Lipschitz truncation $v^{\Lambda}_h$.
\subsection{Proof of Theorem \ref{thm : the Caccioppoli inequality}}
We closely follow the proof of \cite[Subsection 5.1]{Wontae2023a}. For $\tau \in \ell_{S_2-h}(t_0)$ and sufficiently small $\delta>0$, let
$$
\zeta_\delta(t)=
\begin{cases}
    1, &t\in (-\infty,\tau-\delta),\\
    1-\frac{t-\tau+\delta}{\delta}, &t\in[\tau-\delta,\tau],\\
    0, &t\in (\tau,t_0+S_2-h).
\end{cases}
$$
Note that, by Proposition \ref{prop Lip_trunc} \ref{p2-1}, $v_h^\Lambda \eta^s \zeta \zeta_\delta (\cdot,t)\in W_0^{1,\infty}(B_{R_2}(x_0),\mr^N)$ for every $t\in \ell_{S_2-h}(t_0)$. Using this function as the test function of the Steklov averaged weak formulation in \eqref{eq: main equation}, we obtain
$$
\begin{aligned}
    \mathrm{I}+\mathrm{II}&=\iint_{U_{R_2,S_2}(z_0)} \partial_t[u-u_0]_h \cdot v_h^\Lambda \eta^s \zeta \zeta_\delta \, dz\\
    &\qquad +\iint_{U_{R_2,S_2}(z_0)} [\mathcal{A}(z,\nabla u)]_h \cdot \nabla (v_h^\Lambda \eta^s \zeta \zeta_\delta)\, dz\\
    &=\iint_{U_{R_2,S_2}(z_0)} [\mathcal{B}(z,F)]_h \cdot \nabla(v_h^\Lambda \eta^s \zeta \zeta_\delta)\, dz = \mathrm{III}. 
\end{aligned}
$$
Now we estimate the each term above by dividing the integral domain into $E(\Lambda)$ and $E(\Lambda)^c$. 

\noindent {\bf Estimate of $\mathrm{I}$.} By integration by parts and the product rule, we obtain
$$
\begin{aligned}
    \mathrm{I}&=\iint_{U_{R_2,S_2}(z_0)} (-v_h\cdot v_h^\Lambda \eta^{s-1}\partial_t \zeta_\delta -v_h\cdot \partial_t v_h^\Lambda \eta^{s-1}\zeta_\delta) \, dz\\
    &\qquad +\iint_{U_{R_2,S_2}(z_0)} -[u-u_0]_h\cdot v_h^\Lambda \eta^s \zeta_\delta \partial_t \zeta\, dz = \mathrm{I}_1+\mathrm{I}_2.
\end{aligned}
$$
First, we consider $\mathrm{I}_1$. Then we have 
$$
\begin{aligned}
    \mathrm{I}_1&=\iint_{U_{R_2,S_2}(z_0)} -|v_h|^2 \eta^{s-1}\partial_t \zeta_\delta\, dz+ \iint_{U_{R_2,S_2}(z_0)} v_h\cdot(v_h-v_h^\Lambda) \eta^{s-1}\partial_t \zeta_\delta\, dz\\
    &\quad -\iint_{U_{R_2,S_2}(z_0)} (v_h-v_h^\Lambda)\cdot\partial_t v_h^\Lambda \eta^{s-1}\zeta_\delta\, dz- \iint_{U_{R_2,S_2}(z_0)} v_h^\Lambda \cdot\partial_t v_h^\Lambda \eta^{s-1}\zeta_\delta\, dz.
\end{aligned}
$$
Note that integration by parts implies 
$$
\iint_{U_{R_2,S_2}(z_0)} v_h^\Lambda \cdot\partial_t v_h^\Lambda \eta^{s-1}\zeta_\delta\, dz=-\frac{1}{2}\iint_{U_{R_2,S_2}(z_0)} |v_h^\Lambda|^2 \eta^{s-1}\partial_t\zeta_\delta\, dz.
$$
Since $v_h^\Lambda=v_h$ a.e. in $E(\Lambda)$ and $v_h=v_h^\Lambda=0$ in $U_{R_2,S_2}(z_0)^c$, we obtain
$$
\begin{aligned}
    \mathrm{I}_1&=\iint_{U_{R_2,S_2}(z_0)} -(|v_h|^2-\frac{1}{2}|v_h^\Lambda|^2) \eta^{s-1}\partial_t \zeta_\delta\, dz\\
    &\quad + \iint_{E(\Lambda)^c} v_h\cdot(v_h-v_h^\Lambda) \eta^{s-1}\partial_t \zeta_\delta\, dz -\iint_{E(\Lambda)^c} (v_h-v_h^\Lambda)\cdot\partial_t v_h^\Lambda \eta^{s-1}\zeta_\delta\, dz.
\end{aligned}
$$
Letting $h\rightarrow 0^+$, we obtain from the properties of Steklov averages, Proposition \ref{prop Lip_trunc} \ref{p3-1} and Lemma \ref{LEMMA4.9} \ref{lemma5.8 i} that 
$$
\begin{aligned}
    \lim_{h\rightarrow 0^+} \mathrm{I}_1&=\iint_{U_{R_2,S_2}(z_0)} -(|v|^2-\frac{1}{2}|v^\Lambda|^2) \eta^{s-1}\partial_t \zeta_\delta\, dz\\
    &\quad + \iint_{E(\Lambda)^c} v\cdot(v-v^\Lambda) \eta^{s-1}\partial_t \zeta_\delta\, dz -\iint_{E(\Lambda)^c} (v-v^\Lambda)\cdot\partial_t v^\Lambda \eta^{s-1}\zeta_\delta\, dz\\
    &=\mathrm{I}_{11}+\mathrm{I}_{12}+\mathrm{I}_{13}.
\end{aligned}
$$
Since the Lipschitz truncation is done only in the bad set $E(\Lambda)^c$, we get
$$
\begin{aligned}
    \mathrm{I}_{11}&=\iint_{E(\Lambda)} -\frac{1}{2}|v|^2\eta^{s-1}\partial_t\zeta_\delta \, dz-\iint_{E(\Lambda)^c} \left(|v|^2-\frac{1}{2}|v^\Lambda|^c\right)\eta^{s-1}\partial_t\zeta_\delta \, dz
\end{aligned}
$$
Since $v\in L^2(U_{R_2,S_2}(z_0),\mr^N)$, we obtain from the absolute continuity, Lemma \ref{LEMMA4.9} \ref{lemma5.8 iii} and \eqref{limit of Λ|E(Λ)^c|} that
$$
    \lim_{\Lambda\rightarrow\infty}\mathrm{I}_{11}=\iint_{U_{R_2,S_2}(z_0)} -\frac{1}{2}|v|^2\eta^{s-1}\partial_t\zeta_\delta \, dz.
$$
For the same reason, we also have $\displaystyle\lim_{\Lambda\rightarrow\infty}\mathrm{I}_{12}=0$. Finally, Lemma \ref{LEMMA4.9} \ref{lemma5.8 ii} and \eqref{limit of Λ|E(Λ)^c|} imply that $\displaystyle \lim_{\Lambda\rightarrow \infty} \mathrm{I}_{13}=0$. Thus, we have
$$
\lim_{\Lambda\rightarrow\infty}\lim_{h\rightarrow 0^+} \mathrm{I}_1 = \iint_{U_{R_2,S_2}(z_0)} -\frac{1}{2}|v|^2\eta^{s-1}\partial_t\zeta_\delta \, dz.
$$

Now, we estimate $\mathrm{I}_2$. By the properties of Steklov averages and Proposition \ref{prop Lip_trunc} \ref{p3-1} and by dividing into the good and bad sets, we get
$$
\begin{aligned}
    \lim_{h\rightarrow 0}\mathrm{I}_2 &= -\iint_{U_{R_2,S_2}(z_0)} (u-u_0)\cdot v^\Lambda \eta^s \zeta_\delta \partial_t \zeta\, dz\\
    &\geq -\iint_{U_{R_2,S_2}(z_0)\cap E(\Lambda)} |u-u_0|^2 |\partial_t \zeta|\, dz-\iint_{E(\Lambda)^c} |u-u_0||v^\Lambda| |\partial_t \zeta| \, dz.
\end{aligned}
$$
Then H\"{o}lder's inequality implies 
$$
\begin{aligned}
    &\iint_{E(\Lambda)^c} |u-u_0||v^\Lambda| |\partial_t \zeta| \, dz\\
    &\qquad \leq\left(\iint_{U_{R_2,S_2}(z_0)\cap E(\Lambda)^c} |u-u_0|^2 |\partial_t \zeta|^2 \, dz\right)^\frac{1}{2}\left(\iint_{E(\Lambda)^c} |v^\Lambda|^2 \, dz\right)^\frac{1}{2}.
\end{aligned}
$$
Since $|u| \in L^2(\Omega_T)$, the first integral vanishes as $\Lambda\rightarrow\infty$. By Lemma \ref{LEMMA4.9} \ref{lemma5.8 iii} and \eqref{limit of Λ|E(Λ)^c|}, we get
$$
\lim_{\Lambda\rightarrow\infty}\iint_{E(\Lambda)^c} |v^\Lambda|^2 \, dz\leq \lim_{\Lambda\rightarrow\infty}c\Lambda^\frac{2}{p}|E(\Lambda)^c|\leq \lim_{\Lambda\rightarrow\infty}c\Lambda|E(\Lambda)^c|=0.
$$
Thus, we get
$$
\lim_{\Lambda\rightarrow\infty}\lim_{h\rightarrow 0^+}\mathrm{I}_2\geq -\iint_{U_{R_2,S_2}(z_0)} |u-u_0|^2 |\partial_t \zeta|\, dz,
$$
and hence we conclude
$$
\lim_{\Lambda\rightarrow\infty}\lim_{h\rightarrow 0^+}\mathrm{I}\geq \iint_{U_{R_2,S_2}(z_0)} -\frac{1}{2}|v|^2\eta^{s-1}\partial_t\zeta_\delta \, dz-\iint_{U_{R_2,S_2}(z_0)} |u-u_0|^2 |\partial_t \zeta|\, dz.
$$

\noindent{\bf Estimate of $\mathrm{II}$.} As in \cite{Wontae2023a}, we obtain
{ $$
\begin{aligned}
    \lim_{h\rightarrow0^+}\mathrm{II}&=\iint_{U_{R_2,S_2}(z_0)\cap E(\Lambda)} \mathcal{A}(z,\nabla u) \cdot \nabla ((u-u_0) \eta^{s+1} \zeta^2 \zeta_\delta)\, dz\\
    &\quad +\iint_{U_{R_2,S_2}(z_0)\cap E(\Lambda)^c} \mathcal{A}(z,\nabla u) \cdot \nabla (v^\Lambda \eta^s \zeta \zeta_\delta)\, dz=\mathrm{II}_1+\mathrm{II}_2.
\end{aligned}
$$}
Since the integral within the good set does not contain $\Lambda$, letting $\Lambda\rightarrow \infty$, we see from \eqref{eq: growth condition of A} and \eqref{eq : bound of gradients of eta and zeta} that
$$
\begin{aligned}
    &\lim_{\Lambda\rightarrow\infty}\lim_{h\rightarrow0^+} \mathrm{II}_1\\
    &\quad=\iint_{U_{R_2,S_2}(z_0)} (\mathcal{A}(z,\nabla u) \cdot \nabla u) \eta^{s+1} \zeta^2 \zeta_\delta\, dz\\
    &\qquad +\iint_{U_{R_2,S_2}(z_0)} \mathcal{A}(z,\nabla u) \cdot (u-u_0) \nabla(\eta^{s+1}) \zeta^2 \zeta_\delta\, dz\\
    &\quad\geq c \iint_{U_{R_2,S_2}(z_0)} (|\nabla u|^p+a(z)|\nabla u|^q+b(z)|\nabla u|^s) \eta^{s+1} \zeta^2 \zeta_\delta\, dz\\
    &\qquad -c\iint_{U_{R_2,S_2}(z_0)} (|\nabla u|^{p-1}+a(z)|\nabla u|^{q-1}+b(z)|\nabla u|^{s-1}) \frac{|u-u_0|}{R_2-R_1} \eta^{s} \zeta^2 \zeta_\delta\, dz,
\end{aligned}
$$
where $c=c(s,\nu,L)$. Then, Young's inequality with conjugate $\left(\frac{p}{p-1},p\right)$, $\left(\frac{q}{q-1},q\right)$ and $\left(\frac{s}{s-1},s\right)$, respectively, gives
\begin{align*}
    &c\iint_{U_{R_2,S_2}(z_0)} (|\nabla u|^{p-1}+a(z)|\nabla u|^{q-1}+b(z)|\nabla u|^{s-1}) \frac{|u-u_0|}{R_2-R_1} \eta^{s} \zeta^2 \zeta_\delta\, dz\\
    &\quad \leq \frac{\nu}{2}\iint_{U_{R_2,S_2}(z_0)} (|\nabla u|^{p}+a(z)|\nabla u|^{q}+b(z)|\nabla u|^{s}) \eta^{s+1} \zeta^2 \zeta_\delta\, dz\\
    &\qquad +c\iint_{U_{R_2,S_2}(z_0)} \left(\frac{|u-u_0|^p}{(R_2-R_1)^p}+a(z)\frac{|u-u_0|^q}{(R_2-R_1)^q}+b(z)\frac{|u-u_0|^s}{(R_2-R_1)^s}\right)\, dz,
\end{align*}
where $c$ depends only on $p,q,s,\nu$ and $L$. Hence we have
$$
\begin{aligned}
    &\lim_{\Lambda\rightarrow\infty}\lim_{h\rightarrow0^+} \mathrm{II}_1\\
    &\quad\geq \frac{\nu}{2}\iint_{U_{R_2,S_2}(z_0)} (|\nabla u|^{p}+a(z)|\nabla u|^{q}+b(z)|\nabla u|^{s}) \eta^{s+1} \zeta^2 \zeta_\delta\, dz\\
    &\qquad -c\iint_{U_{R_2,S_2}(z_0)} \left(\frac{|u-u_0|^p}{(R_2-R_1)^p}+a(z)\frac{|u-u_0|^q}{(R_2-R_1)^q}+b(z)\frac{|u-u_0|^s}{(R_2-R_1)^s}\right)\, dz.
\end{aligned}
$$
Moreover, it is easy to show $\displaystyle \lim_{\Lambda\rightarrow\infty} \lim_{h\rightarrow0^+} \mathrm{II}_2=0$ by using \eqref{eq: growth condition of A}, Young's inequality, Lemma \ref{LEMMA4.9} \ref{lemma5.8 iii} and \eqref{limit of Λ|E(Λ)^c|}. For detailed calculations, refer to \cite{Wontae2023a}. Thus, we conclude
$$
\begin{aligned}
    \lim_{\Lambda\rightarrow\infty}\lim_{h\rightarrow0^+} \mathrm{II} &\geq \frac{\nu}{2}\iint_{U_{R_2,S_2}(z_0)} H(z,|\nabla u|) \eta^{s+1} \zeta^2 \zeta_\delta\, dz\\
    &\qquad -c\iint_{U_{R_2,S_2}(z_0)} H\left(\frac{|u-u_0|}{R_2-R_1}\right)\, dz.
\end{aligned}
$$

\noindent{\bf Estimate of $\mathrm{III}.$} The estimate for $\mathrm{III}$ follows a process similar to that for $\mathrm{II}$. Again, we divide into the good and bad parts to obtain
$$
\begin{aligned}
    \lim_{h\rightarrow0^+} \mathrm{III}&=\iint_{U_{R_2,S_2}(z_0)\cap E(\Lambda)} \mathcal{B}(z,F) \cdot \nabla((u-u_0) \eta^{s+1} \zeta^2 \zeta_\delta)\, dz\\
    &\quad +\iint_{U_{R_2,S_2}(z_0)\cap E(\Lambda)^c} \mathcal{B}(z,F) \cdot \nabla(v_h^\Lambda \eta^s \zeta \zeta_\delta)\, dz=\mathrm{III}_1+\mathrm{III}_2.
\end{aligned}
$$
Applying \eqref{eq: growth condition of B} and Young's inequality, we get
$$
\begin{aligned}
    \lim_{\Lambda\rightarrow \infty}\lim_{h\rightarrow0^+} \mathrm{III}_1&\leq c\iint_{U_{R_2,S_2}(z_0)} \left(H\left(z,\frac{|u-u_0|}{R_2-R_1}\right)+H(z,|F|)\right)\, dz\\
    &\quad +\frac{\nu}{4}\iint_{U_{R_2,S_2}(z_0)}H(z,|\nabla u|) \eta^{s+1} \zeta^2 \zeta_\delta)\, dz
\end{aligned}
$$
for some $c=c(p,q,s,\nu,L)$. In the same reason as $\mathrm{II}_2$, we get $\displaystyle \lim_{\Lambda\rightarrow\infty}\lim_{h\rightarrow0^+}\mathrm{III}_2=0$. 
Thus, we conclude that 
$$
\begin{aligned}
    \lim_{\Lambda\rightarrow \infty}\lim_{h\rightarrow0^+} \mathrm{III}&\leq c\iint_{U_{R_2,S_2}(z_0)} \left(H\left(z,\frac{|u-u_0|}{R_2-R_1}\right)+H(z,|F|)\right)\, dz\\
    &\quad +\frac{\nu}{4}\iint_{U_{R_2,S_2}(z_0)}H(z,|\nabla u|) \eta^{s+1} \zeta^2 \zeta_\delta)\, dz.
\end{aligned}
$$
Combining all the estimates for $\mathrm{I}$, $\mathrm{II}$ and $\mathrm{III}$, we obtain
$$
\begin{aligned}
    &\iint_{U_{R_2,S_2}(z_0)} -\frac{1}{2}|v|^2\eta^{s-1}\partial_t\zeta_\delta \, dz+\frac{\nu}{4}\iint_{U_{R_2,S_2}(z_0)}H(z,|\nabla u|) \eta^{s+1} \zeta^2 \zeta_\delta)\, dz\\
    &\quad\leq c\iint_{U_{R_2,S_2}(z_0)} \left(H\left(z,\frac{|u-u_0|}{R_2-R_1}\right)+|u-u_0|^2 |\partial_t \zeta|+H(z,|F|)\right)\, dz.
\end{aligned}
$$
Finally, we complete the proof by letting $\delta \rightarrow 0$, recalling that $\tau\in \ell_{S_1}(t_0)$ is arbitrary, and replacing $u_0$ with $(u)_{U_{R_1,S_1}(z_0)}$. \qed

${}$

{
{\bf Acknowledgments.} The authors would like to express their sincere gratitude to the anonymous referee who provided valuable comments and suggestions on the earlier version, which have greatly improved the quality and clarity of the manuscript.}

\bibliographystyle{abbrv}
\bibliography{ref}{}

\begin{thebibliography}{10}

\bibitem{Acerbi1984}
E.~Acerbi and N.~Fusco.
\newblock Semicontinuity problems in the calculus of variations.
\newblock {\em Arch. Rational Mech. Anal.}, 86(2):125--145, 1984.

\bibitem{Acerbi1988}
E.~Acerbi and N.~Fusco.
\newblock An approximation lemma for {$W^{1,p}$} functions.
\newblock In {\em Material instabilities in continuum mechanics ({E}dinburgh,
  1985--1986)}, Oxford Sci. Publ., pages 1--5. Oxford Univ. Press, New York,
  1988.

\bibitem{Baasandorj2020}
S.~Baasandorj, S.-S. Byun, and J.~Oh.
\newblock Calder\'{o}n-{Z}ygmund estimates for generalized double phase
  problems.
\newblock {\em J. Funct. Anal.}, 279(7):108670, 57, 2020.

\bibitem{Baasandorj2021}
S.~Baasandorj, S.-S. Byun, and J.~Oh.
\newblock Gradient estimates for multi-phase problems.
\newblock {\em Calc. Var. Partial Differential Equations}, 60(3):Paper No. 104,
  48, 2021.

\bibitem{Baroni2015}
P.~Baroni, M.~Colombo, and G.~Mingione.
\newblock Harnack inequalities for double phase functionals.
\newblock {\em Nonlinear Anal.}, 121:206--222, 2015.

\bibitem{Baroni2018}
P.~Baroni, M.~Colombo, and G.~Mingione.
\newblock Regularity for general functionals with double phase.
\newblock {\em Calc. Var. Partial Differential Equations}, 57(2):Paper No. 62,
  48, 2018.

\bibitem{Byun2021a}
S.-S. Byun and H.-S. Lee.
\newblock Calder\'{o}n-{Z}ygmund estimates for elliptic double phase problems
  with variable exponents.
\newblock {\em J. Math. Anal. Appl.}, 501(1):Paper No. 124015, 31, 2021.

\bibitem{Byun2021}
S.-S. Byun and H.-S. Lee.
\newblock Gradient estimates of {$\omega$}-minimizers to double phase
  variational problems with variable exponents.
\newblock {\em Q. J. Math.}, 72(4):1191--1221, 2021.

\bibitem{Byun2017}
S.-S. Byun and J.~Oh.
\newblock Global gradient estimates for non-uniformly elliptic equations.
\newblock {\em Calc. Var. Partial Differential Equations}, 56(2):Paper No. 46,
  36, 2017.

\bibitem{Byun2020}
S.-S. Byun and J.~Oh.
\newblock Regularity results for generalized double phase functionals.
\newblock {\em Anal. PDE}, 13(5):1269--1300, 2020.

\bibitem{V_Boglein_phd_thesis}
V.~Bögelein.
\newblock {\em Regularity results for weak and very weak solutions of higher
  order parabolic systems}.
\newblock PhD thesis, Friedrich-Alexander-Universität Erlangen-Nürnberg (FAU)
  Naturwissenschaftliche Fakultät, 2007.

\bibitem{Chlebicks2019}
I.~Chlebicka, P.~Gwiazda, and A.~Zatorska-Goldstein.
\newblock Parabolic equation in time and space dependent anisotropic
  {M}usielak-{O}rlicz spaces in absence of {L}avrentiev's phenomenon.
\newblock {\em Ann. Inst. H. Poincar\'{e} C Anal. Non Lin\'{e}aire},
  36(5):1431--1465, 2019.

\bibitem{Colombo2015a}
M.~Colombo and G.~Mingione.
\newblock Bounded minimisers of double phase variational integrals.
\newblock {\em Arch. Ration. Mech. Anal.}, 218(1):219--273, 2015.

\bibitem{Colombo2015}
M.~Colombo and G.~Mingione.
\newblock Regularity for double phase variational problems.
\newblock {\em Arch. Ration. Mech. Anal.}, 215(2):443--496, 2015.

\bibitem{Colombo2016}
M.~Colombo and G.~Mingione.
\newblock Calder\'{o}n-{Z}ygmund estimates and non-uniformly elliptic
  operators.
\newblock {\em J. Funct. Anal.}, 270(4):1416--1478, 2016.

\bibitem{DeFilippis2022}
C.~De~Filippis.
\newblock Optimal gradient estimates for multi-phase integrals.
\newblock {\em Math. Eng.}, 4(5):Paper No. 043, 36, 2022.

\bibitem{DeFilippis2019}
C.~De~Filippis and G.~Mingione.
\newblock A borderline case of {C}alder\'{o}n-{Z}ygmund estimates for
  nonuniformly elliptic problems.
\newblock {\em St. Petersburg Mathematical J.}, 31(3):82--115, 2019.

\bibitem{DeFilippis2019a}
C.~De~Filippis and J.~Oh.
\newblock Regularity for multi-phase variational problems.
\newblock {\em J. Differential Equations}, 267(3):1631--1670, 2019.

\bibitem{Esposito2004}
L.~Esposito, F.~Leonetti, and G.~Mingione.
\newblock Sharp regularity for functionals with {$(p,q)$} growth.
\newblock {\em J. Differential Equations}, 204(1):5--55, 2004.

\bibitem{Fang2022}
Y.~Fang, V.~D. R\u{a}dulescu, C.~Zhang, and X.~Zhang.
\newblock Gradient estimates for multi-phase problems in {C}ampanato spaces.
\newblock {\em Indiana Univ. Math. J.}, 71(3):1079--1099, 2022.

\bibitem{Fonseca2004}
I.~Fonseca, J.~Mal\'{y}, and G.~Mingione.
\newblock Scalar minimizers with fractal singular sets.
\newblock {\em Arch. Ration. Mech. Anal.}, 172(2):295--307, 2004.

\bibitem{Haestoe2022a}
P.~H\"{a}st\"{o} and J.~Ok.
\newblock Maximal regularity for local minimizers of non-autonomous
  functionals.
\newblock {\em J. Eur. Math. Soc. (JEMS)}, 24(4):1285--1334, 2022.

\bibitem{Haestoe2022}
P.~H\"{a}st\"{o} and J.~Ok.
\newblock Regularity theory for non-autonomous partial differential equations
  without {U}hlenbeck structure.
\newblock {\em Arch. Ration. Mech. Anal.}, 245(3):1401--1436, 2022.

\bibitem{Kim_Oh_2024}
B.~Kim and J.~Oh.
\newblock Higher integrability for weak solutions to parabolic multi-phase
  equations.
\newblock {\em J. Differential Equations}, 409:223--298, 2024.

\bibitem{Oh2024a}
B.~Kim and J.~Oh.
\newblock Regularity for double phase functionals with two modulating
  coefficients.
\newblock {\em J. Geom. Anal.}, 34(5):Paper No. 134, 51, 2024.

\bibitem{Wontae2023b}
W.~Kim.
\newblock Calder\'{o}n-{Z}ygmund type estimate for the parabolic double-phase
  system.
\newblock {\em To apper, Ann. Sc. Norm. Super. Pisa Cl. Sci. (5)}, 2024.

\bibitem{2023_Gradient_Higher_Integrability_for_Degenerate_Parabolic_Double-Phase_Systems}
W.~Kim, J.~Kinnunen, and K.~Moring.
\newblock Gradient higher integrability for degenerate parabolic double-phase
  systems.
\newblock {\em Arch. Ration. Mech. Anal.}, 247(5):Paper No. 79, 46, 2023.

\bibitem{Wontae2023a}
W.~Kim, J.~Kinnunen, and L.~S{\"a}rki{\"o}.
\newblock Lipschitz truncation method for the parabolic double-phase system and
  applications.
\newblock {\em To appear, J. Funct. Anal.}, 2024.

\bibitem{Kinnunen2021}
J.~Kinnunen, J.~Lehrb\"ack, and A.~V\"ah\"akangas.
\newblock {\em Maximal function methods for {S}obolev spaces}, volume 257 of
  {\em Mathematical Surveys and Monographs}.
\newblock American Mathematical Society, Providence, RI, [2021] \copyright
  2021.

\bibitem{Kinnunen2002}
J.~Kinnunen and J.~L. Lewis.
\newblock Very weak solutions of parabolic systems of {$p$}-{L}aplacian type.
\newblock {\em Ark. Mat.}, 40(1):105--132, 2002.

\bibitem{Ok2017}
J.~Ok.
\newblock Regularity of {$\omega$}-minimizers for a class of functionals with
  non-standard growth.
\newblock {\em Calc. Var. Partial Differential Equations}, 56(2):Paper No. 48,
  31, 2017.

\bibitem{Ok2020}
J.~Ok.
\newblock Regularity for double phase problems under additional integrability
  assumptions.
\newblock {\em Nonlinear Anal.}, 194:111408, 13, 2020.

\bibitem{sen2024}
A.~Sen.
\newblock Gradient higher integrability for degenerate/ singular parabolic
  multi-phase problems.
\newblock {\em https://arxiv.org/abs/2406.00763v2}, 2024.

\bibitem{Singer2016}
T.~Singer.
\newblock Existence of weak solutions of parabolic systems with {$p,
  q$}-growth.
\newblock {\em Manuscripta Math.}, 151(1-2):87--112, 2016.

\bibitem{Zhikov1986}
V.~V. Zhikov.
\newblock Averaging of functionals of the calculus of variations and elasticity
  theory.
\newblock {\em Izv. Akad. Nauk SSSR Ser. Mat.}, 50(4):675--710, 877, 1986.

\bibitem{Zhikov1993}
V.~V. Zhikov.
\newblock Lavrentiev phenomenon and homogenization for some variational
  problems.
\newblock {\em C. R. Acad. Sci. Paris S\'{e}r. I Math.}, 316(5):435--439, 1993.

\bibitem{Zhikov1995}
V.~V. Zhikov.
\newblock On {L}avrentiev's phenomenon.
\newblock {\em Russian J. Math. Phys.}, 3(2):249--269, 1995.

\bibitem{Zhikov1997}
V.~V. Zhikov.
\newblock On some variational problems.
\newblock {\em Russian J. Math. Phys.}, 5(1):105--116 (1998), 1997.

\end{thebibliography}
\end{document}